\documentclass[oneside,english]{amsart}
\usepackage{lmodern}
\usepackage[T1]{fontenc}
\usepackage[latin9]{inputenc}
\usepackage{color}
\usepackage{babel}
\usepackage{amsthm}
\usepackage{amssymb}
\usepackage[unicode=true,pdfusetitle,
 bookmarks=true,bookmarksnumbered=true,bookmarksopen=true,bookmarksopenlevel=1,
 breaklinks=false,pdfborder={0 0 0},backref=false,colorlinks=true]
 {hyperref}
\usepackage{breakurl}

\makeatletter
\numberwithin{equation}{section}
\numberwithin{figure}{section}
\theoremstyle{plain}
\newtheorem{thm}{\protect\theoremname}[section]
  \theoremstyle{definition}
  \newtheorem{defn}[thm]{\protect\definitionname}
  \theoremstyle{definition}
  \newtheorem{example}[thm]{\protect\examplename}
  \theoremstyle{plain}
  \newtheorem{lem}[thm]{\protect\lemmaname}
  \theoremstyle{remark}
  \newtheorem{rem}[thm]{\protect\remarkname}
  \theoremstyle{plain}
  \newtheorem{prop}[thm]{\protect\propositionname}
  \theoremstyle{plain}
  \newtheorem{cor}[thm]{\protect\corollaryname}

\usepackage[a4paper]{geometry}
\allowdisplaybreaks[2]

\usepackage{stmaryrd}
\usepackage{amsfonts}
\usepackage{bbm}
\usepackage{amscd}
\usepackage{mathrsfs}
\usepackage{latexsym}
\usepackage{amscd}
\usepackage{amscd}
\usepackage{amsthm}
\usepackage{amsxtra}
\usepackage{xypic}
\usepackage{lmodern}
\usepackage[all]{xy}
\usepackage{nicefrac}
\usepackage{mathtools}
\usepackage{enumitem}
\usepackage{microtype}
\usepackage{pdfpages}
\usepackage{changepage}
\usepackage[toc,page,title,titletoc,header]{appendix} 

\newcommand{\be }{\begin{equation}}
\newcommand{\ee }{\end{equation}}

\newcommand {\emptycomment}[1]{} %to remove paragraphs

\newcommand{\br}[1]{   [ \cdot,    \cdot  ]   }

\makeatother

  \providecommand{\corollaryname}{Corollary}
  \providecommand{\definitionname}{Definition}
  \providecommand{\examplename}{Example}
  \providecommand{\lemmaname}{Lemma}
  \providecommand{\propositionname}{Proposition}
  \providecommand{\remarkname}{Remark}
\providecommand{\theoremname}{Theorem}

\begin{document}

\title{$H$-standard cohomology for Courant-Dorfman algebras and Leibniz
algebras}

\author{Xiongwei Cai}

\address{Mathematics Research Unit, FSTC, University of Luxembourg, Luxembourg}

\email{shernvey@gmail.com}

\begin{abstract}
We introduce the notion of $H$-standard cohomology for Courant-Dorfman
algebras and Leibniz algebras, by generalizing Roytenberg's construction.
Then we generalize a theorem of Ginot-Grutzmann on transitive Courant
algebroids, which was conjectured by Stienon-Xu. The relation between
$H$-standard complexes of a Leibniz algebra and the associated crossed
product is also discussed. 
\end{abstract}
\maketitle

\section{Introduction}

The notion of Leibniz algebras, objects that date back to the work
of ``D-algebras'' by Bloh \cite{Bloh65}, is due to Loday \cite{Loday93}.
In literature, Leibniz algebras are sometimes also called Loday algebras.

A (left) Leibniz algebra $L$ is a vector space over a field $\mathfrak{k}$
($\mathfrak{k}=\mathbb{R}$ or $\mathbb{C}$) equipped with a bracket
$\circ:L\otimes L\rightarrow L$, called the Leibniz bracket, satisfying
the (left) Leibniz identity:
\[
x\circ(y\circ z)=(x\circ y)\circ z+y\circ(x\circ z),\qquad\forall x,y,z\in L.
\]

A concrete example is the omni Lie algebra $ol(V)\triangleq gl(V)\oplus V$,
where $V$ is a vector space. It is first introduced by Weinstein
\cite{Weinstein} as the linearization of the standard Courant algebroid
$TV^{*}\oplus T^{*}V^{*}$. The Leibniz bracket of $ol(V)$ is given
by:
\[
(A+v)\circ(B+w)=[A,B]+Aw,\quad\forall A,B\in gl(V),\ v,w\in V.
\]

In \cite{LodayPirash93}, Loday and Pirashvili introduced the notions
of representations (corepresentations) and Leibniz homology (cohomology)
for Leibniz algebras. They also studied universal enveloping algebras
and PBW theorem for Leibniz algebras.

Leibniz algebras can be viewed as a non-commutative analogue of Lie
algebras. Some theorems and properties of Lie algebras are still valid
for Leibniz algebras, while some are not. The properties of Leibniz
algebras are under continuous investigation by many authors, we can
only mention a few works here \cite{Barnes,Covez,Patsourakos,ShengLiu13,ShengLiu16}.

Leibniz algebras have attracted more interest since the discovery
of Courant algebroids, which can be viewed as the geometric realization
of Leibniz algebras in certain sense. Courant algebroids are important
objects in recent studies of Poisson geometry, symplectic geometry
and generalized complex geometry.

The notion of Courant algebroids was first introduced by Liu, Weistein
and Xu in \cite{LiuWX}, as an answer to an earlier question ``what
kind of object is the double of a Lie bialgebroid''. In their original
definition, a Courant algebroid is defined in terms of a skew-symmetric
bracket, now known as ``Courant bracket''. In \cite{RoytPhD}, Roytenberg
proved that a Courant algebroid can be equivalently defined in terms
of a Leibniz bracket, now known as ``Dorfman bracket''. And he defined
standard complex and standard cohomology of Courant algebroids in
the language of supermanifolds in \cite{RoytNQmfd}. In \cite{StienonXu08},
Stienon and Xu defined naive cohomology of Courant algebroids , and
conjectured that there is an isomorphism between standard and naive
cohomology for a transitive Courant algebroid. Later this conjecture
was proved by Ginot and Grutzmann in \cite{GinotGrutz08}. 

In \cite{RoytCDA}, Roytenberg introduced the notion of Courant-Dorfman
algebras, as an algebraic analogue of Courant algebroids. And he defined
standard complex and standard cohomology for Courant-Dorfman algebras.
Furthermore, he proved that there is an isomorphism of graded Poisson
algebras between the standard complex of a Courant algebroid $E$
and the associated Courant-Dorfman algebra $\mathcal{E}=\Gamma(E)$.

The main objective of this article is to develop a similar cohomology
theory, the so-called $H$-standard cohomology, for Courant-Dorfman
algebras as well as Leibniz algebras.

Given a Courant-Dorfman algebra $(\mathcal{E},R,\langle\cdot,\cdot\rangle,\partial,\circ)$,
and an $R$-submodule $\mathcal{H}\supseteq\partial R$ which is an
isotropic ideal of $\mathcal{E}$, let $(\mathcal{V},\nabla)$ be
an $\mathcal{H}$-representation of $\mathcal{E}$ (a left representation
of Leibniz algebra $\mathcal{E}$ such that $\nabla$ is a covariant
differential and $\mathcal{H}$ acts trivially on $\mathcal{V}$).
By generalizing Roytenberg's construction, we shall define the $\mathcal{H}$-standard
complex $(C^{\bullet}(\mathcal{E},\mathcal{H},\mathcal{V}),d)$ and
$\mathcal{H}$-standard cohomology $H^{\bullet}(\mathcal{E},\mathcal{H},\mathcal{V})$
(see Theorem \ref{thm:standard complex CDA} and Definition \ref{Def:generalized standard cohomology CDA}).
And when $\mathcal{E}/\mathcal{H}$ is projective, we shall prove
the following result:
\[
H^{\bullet}(\mathcal{E},\mathcal{H},\mathcal{V})\cong H_{CE}^{\bullet}(\mathcal{E}/\mathcal{H},\mathcal{V}).
\]
Note that we don't require the symmetric bilinear form $\langle\cdot,\cdot\rangle$
to be non-degenerate here. In particular when $\mathcal{E}$ is the
space of sections of a transitive Courant algebroid $E$ (over $M$),
and $\mathcal{H}=\rho^{*}(\Omega^{1}(M)),\ \mathcal{V}=C^{\infty}(M)$,
the result above recovers Stienon and Xu's conjecture.

Given a Leibniz algebra $L$ with left center $Z$, suppose $H\supseteq Z$
is an isotropic ideal of $L$, and $(V,\tau)$ is an $H$-representation
of $L$ (a left representation of $L$ such that $H$ acts trivially
on $V$), similarly we can define the $H$-standard complex $(C^{\bullet}(L,H,V),d)$
and $H$-standard cohomology $H^{\bullet}(L,H,V)$. And we have the
following result:
\[
H^{\bullet}(L,H,V)\cong H_{CE}^{\bullet}(L/H,V).
\]
This result can be proved directly, but in this paper we choose a
roundabout way. We construct a Courant-Dorfman algebra structure on
$\mathcal{L}=S^{\bullet}(Z)\otimes L$, and then prove there is an
isomorphism between the $H$-standard complex of $L$ and the $\mathcal{H}=S^{\bullet}(Z)\otimes H$-standard
complex of $\mathcal{L}$. Finally based on the result for Courant-Dorfman
algebras, we may obtain the result above by inference.

The structure of this paper is organized as follows:

In Section 2, we provide some basic knowledge about Leibniz algebras
and Courant-Dorfman algebras. In Section 3, we give the definition
of $H$-standard complex and cohomology for Courant-Dorfman algebras
and Leibniz algebras. In Section 4, we prove the isomorphism theorem
for Courant-Dorfman algebras, as a generalization of Stienon and Xu's
conjecture. In Section 5, we associate a Courant-Dorfman algebra structure
on $\mathcal{L}$ to any Leibniz algebra $L$, and discuss the relation
between $H$-standard complexes of them, finally we prove an isomorphism
theorem for Leibniz algebras.

\subsection*{Acknowledgements}

This paper is a part of my PhD dissertation, it is funded by the University
of Luxembourg. I would like to thank my advisors, Prof. Martin Schlichenmaier
and Prof. Ping Xu, for their continual encouragement and support.
I am particularly grateful to Prof. Zhangju Liu for instructive discussions
and helpful comments during my stay in Peking University.

\section{Preliminaries}

In this section we list some basic notions and properties about Leibniz
algebras and Courant-Dorfman algebras. For more details we refer to
\cite{LodayPirash93,RoytCDA}.
\begin{defn}
A (left) Leibniz algebra is a vector space $L$ over a field $\mathfrak{k}$
($\mathfrak{k}=\mathbb{R}$ for our main interest), endowed with a
bilinear map (called Leibniz bracket) $\circ:\ L\otimes L\rightarrow L$,
which satisfies (left) Leibniz rule:
\[
e_{1}\circ(e_{2}\circ e_{3})=(e_{1}\circ e_{2})\circ e_{3}+e_{2}\circ(e_{1}\circ e_{3})\quad\forall e_{1},e_{2},e_{3}\in L
\]
\end{defn}
\begin{example}
Given any Lie algebra $g$ and its representation $(V,\rho)$, the
semi-direct product $g\ltimes V$ with a bilinear operation $\circ$
defined by
\[
(A+v)\circ(B+w)\triangleq[A,B]+\rho(A)w,\quad\forall A,B\in g,\ v,w\in V
\]
 forms a Leibniz algebra.\label{Eg:Leibniz algebra}

In particular, for any vector space $V$, $gl(V)\oplus V$ is a Leibniz
algebra with Leibniz bracket
\[
(A+v)\circ(B+w)=[A,B]+Aw,\quad\forall A,B\in gl(V),\ v,w\in V.
\]
It is called an omni Lie algebra, and denoted by $ol(V)$ (see Weinstein
\cite{Weinstein}).\end{example}
\begin{defn}
A representation of a Leibniz algebra $L$ is a triple $(V,l,r)$,
where $V$ is a vector space equipped with two linear maps: left action
$l:\ L\rightarrow gl(V)$ and right action $r:\ L\rightarrow gl(V)$
satisfying the following equations:
\begin{equation}
l_{e_{1}\circ e_{2}}=[l_{e_{1}},l_{e_{2}}],\ r_{e_{1}\circ e_{2}}=[l_{e_{1}},r_{e_{2}}],\ r_{e_{1}}\circ l_{e_{2}}=-r_{e_{1}}\circ r_{e_{2}},\quad\forall e_{1},e_{2}\in L,\label{eq:Leibniz module}
\end{equation}

where the brackets on the right hand side are the commutators in $gl(V)$.

If $V$ is only equipped with left action $l:\ L\rightarrow gl(V)$
which satisfies $l_{e_{1}\circ e_{2}}=[l_{e_{1}},l_{e_{2}}]$, we
call $(V,l)$ a left representation of $L$.

For $(V,l,r)$(or $(V,l)$) a representation (or left representation)
of $L$, we call $V$ an $L$-module (or left $L$-module).
\end{defn}
Given a left representation $(V,l)$, there are two standard ways
to extend $V$ to an $L$-module. One is called symmetric extension,
with the right action defined as $r_{e}=-l_{e}$; the other is called
antisymmetric extension, with the right action defined as $r_{e}=0$.
In this paper, we always take the symmetric extension $(V,l,-l)$.
\begin{example}
Denote by $Z$ the left center of $L$, i.e.
\[
Z\triangleq\{e\in L|e\circ e^{\prime}=0,\ \forall e^{\prime}\in L\}.
\]
It is easily checked that 
\[
e_{1}\circ e_{2}+e_{2}\circ e_{1}\in Z,\qquad\forall e_{1},e_{2}\in L
\]
and $Z$ is an ideal of $L$. Moreover, the Leibniz bracket of $L$
induces a left action $\rho$ of $L$ on $Z$:
\[
\rho(e)z\triangleq e\circ z\qquad\forall e\in L,z\in Z.
\]
\end{example}
\begin{defn}
Given a Leibniz algebra $L$ and an $L$-module $(V,l,r)$, the Leibniz
cohomology of $L$ with coefficients in $V$ is the cohomology of
the cochain complex $C^{n}(L,V)=Hom(\otimes^{n}L,V)\ (n\geq0)$ with
the coboundary operator $d_{0}:\ C^{n}(L,V)\rightarrow C^{n+1}(L,V)$
given by:
\begin{eqnarray}
 &  & (d_{0}\eta)(e_{1},\cdots,e_{n+1})\nonumber \\
 & = & \sum_{a=1}^{n}(-1)^{a+1}l_{e_{a}}\eta(e_{1},\cdots,\hat{e_{a}},\cdots,e_{n+1})+(-1)^{n+1}r_{e_{n+1}}\eta(e_{1},\cdots,e_{n})\nonumber \\
 &  & +\sum_{1\leq a<b\leq n+1}(-1)^{a}\eta(e_{1},\cdots,\hat{e_{a}},\cdots,\hat{e_{b}},e_{a}\circ e_{b},\cdots,e_{n+1})\label{eq: Leibniz coboundary}
\end{eqnarray}

The resulting cohomology is denoted by $H^{\bullet}(L;V,l,r)$, or
simply $H^{\bullet}(L,V)$ if it causes no confusion.
\end{defn}
As a special type of Leibniz algebras, Courant-Dorfman algebras can
be viewed as the algebraization of Courant algebroids:
\begin{defn}
\label{Def:Courant-Dorfman-algebra} A Courant-Dorfman algebra $(\mathcal{E},R,\langle\cdot,\cdot\rangle,\partial,\circ)$
consists of the following data: 

a commutative algebra $R$ over a field $\mathfrak{k}$ ($\mathfrak{k}=\mathbb{R}$
for our main interest);

an $R$-module $\mathcal{E}$;

a symmetric bilinear form $\langle\cdot,\cdot\rangle:\ \mathcal{E}\otimes_{R}\mathcal{E}\rightarrow R$;

a derivation $\partial:\ R\rightarrow\mathcal{E}$;

a Dorfman bracket $\circ:\ \mathcal{E}\otimes\mathcal{E}\rightarrow\mathcal{E}$.

These data are required to satisfy the following conditions for any
$e,e_{1},e_{2},e_{3}\in\mathcal{E}$ and $f,g\in R$:

(1). $e_{1}\circ(fe_{2})=f(e_{1}\circ e_{2})+\langle e_{1},\partial f\rangle e_{2}$;

(2). $\langle e_{1},\partial(e_{2},e_{3})\rangle=\langle e_{1}\circ e_{2},e_{3}\rangle+\langle e_{2},e_{1}\circ e_{3}\rangle$

(3). $e_{1}\circ e_{2}+e_{2}\circ e_{1}=\partial\langle e_{1},e_{2}\rangle$;

(4). $e_{1}\circ(e_{2}\circ e_{3})=(e_{1}\circ e_{2})\circ e_{3}+e_{2}\circ(e_{1}\circ e_{3})$;

(5). $\partial f\circ e=0$;

(6). $\langle\partial f,\partial g\rangle=0.$
\end{defn}
Given a Courant-Dorfman algebra $\mathcal{E}$, we can recover the
anchor map 
\[
\rho:\ \mathcal{E}\rightarrow\mathfrak{X}^{1}=Der(R,R)
\]
 from the derivation $\partial$ by setting:
\begin{equation}
\rho(e)\cdot f\triangleq\langle e,\partial f\rangle.\label{eq:anchor CDA}
\end{equation}

Let $\Omega^{1}$ be the $R$-module of Kahler differentials with
the universal derivation $d_{R}:R\rightarrow\Omega^{1}$. By the universal
property of $\Omega^{1}$, there is a unique homomorphism of $R$-modules
$\rho^{*}:\Omega^{1}\rightarrow\mathcal{E}$ such that
\begin{equation}
\rho^{*}(d_{R}f)\triangleq\partial f,\qquad\forall f\in R\label{eq:coanchor CDA}
\end{equation}
 $\rho^{*}$ is called the coanchor map of $\mathcal{E}$. When the
bilinear form of $\mathcal{E}$ is non-degenerate, $\rho^{*}$ can
be equivalently defined by
\[
\langle\rho^{*}\alpha,e\rangle=\langle\alpha,\rho(e)\rangle,\qquad\forall\alpha\in\Omega^{1},e\in\mathcal{E},
\]
where $\langle\cdot,\cdot\rangle$ on the right handside is the natural
pairing of $\Omega^{1}$ and $\mathfrak{X}^{1}$.

In the following of this section, we always assume
\[
e\in\mathcal{E},\ \alpha\in\Omega^{1},\ f\in R.
\]

Given a Courant-Dorfman algebra $\mathcal{E}$, denote by $C^{n}(\mathcal{E},R)$
the space of all sequences $\omega=(\omega_{0},\cdots,\omega_{[\frac{n}{2}]})$,
where $\omega_{k}$ is a linear map from $(\otimes^{n-2k}\mathcal{E})\otimes(\odot^{k}\Omega^{1})$
to $R$, $\forall k$, satisfying the following conditions:

1). Weak skew-symmetricity in arguments of $\mathcal{E}$.

$\forall k$, $\omega_{k}$ is weakly skew-symmetric up to $\omega_{k+1}$:
\begin{eqnarray*}
 &  & \omega_{k}(e_{1},\cdots e_{a},e_{a+1},\cdots e_{n-2k};\alpha_{1},\cdots\alpha_{k})+\omega_{k}(e_{1},\cdots e_{a+1},e_{a},\cdots e_{n-2k};\alpha_{1},\cdots\alpha_{k})\\
 & = & -\omega_{k+1}(e_{1},\cdots\widehat{e_{a}},\widehat{e_{a+1}},\cdots e_{n-2k};d_{R}\langle e_{a},e_{b}\rangle,\alpha_{1},\cdots\alpha_{k}),
\end{eqnarray*}

2). Weak $R$-linearity in arguments of $\mathcal{E}$.

$\forall k$, $\omega_{k}$ is weakly $R$-linear up to $\omega_{k+1}$:
\begin{eqnarray*}
 &  & \omega_{k}(e_{1},\cdots fe_{a},\cdots e_{n-2k};\alpha_{1},\cdots\alpha_{k})\\
 & = & f\omega_{k}(e_{1},\cdots e_{a},\cdots e_{n-2k};\cdots)+\sum_{b>a}(-1)^{b-a}\langle e_{a},e_{b}\rangle\omega_{k+1}(e_{1},\cdots\widehat{e_{a}},\cdots\widehat{e_{b}},\cdots e_{n-2k};d_{R}f,\cdots),
\end{eqnarray*}

3). $R$-linearity in arguments of $\Omega^{1}$.

$\forall k$, $\omega_{k}$ is $R$-linear in arguments of $\Omega^{1}$:
\[
\omega_{k}(e_{1},\cdots e_{n-2k};\alpha_{1},\cdots f\alpha_{l},\cdots\alpha_{k})=f\omega_{k}(e_{1},\cdots e_{n-2k};\alpha_{1},\cdots\alpha_{l},\cdots\alpha_{k}),
\]

Then $C^{\bullet}(\mathcal{E},R)=\bigoplus_{n}C^{n}(\mathcal{E},R)$
becomes a cochain complex, with coboundary map $d$ given for any
$\omega\in C^{n}(\mathcal{E},R)$ by:

\begin{eqnarray*}
 &  & (d\omega)_{k}(e_{1},\cdots,e_{n+1-2k};\alpha_{1},\cdots,\alpha_{k})\\
 & = & \sum_{a}(-1)^{a+1}\rho(e_{a})\omega_{k}(\cdots\widehat{e_{a}},\cdots;\cdots)+\sum_{a<b}(-1)^{a}\omega_{k}(\cdots\widehat{e_{a}},\cdots\widehat{e_{b}},e_{a}\circ e_{b},\cdots;\cdots)\\
 &  & +\sum_{i}\omega_{k-1}(\rho^{*}(\alpha_{i}),e_{1},\cdots,e_{n+1-2k};\alpha_{1},\cdots\widehat{\alpha_{i}},\cdots\alpha_{k})\\
 &  & +\sum_{a,i}(-1)^{a}\omega_{k}(\cdots,\widehat{e_{a}},\cdots;\cdots,\widehat{\alpha_{i}},\iota_{\rho(e_{a})}d_{R}\alpha_{i},\cdots).
\end{eqnarray*}

\begin{defn}
\label{Def:standard cohomology CDA}The cochain complex $(C^{\bullet}(\mathcal{E},R),d)$
is called the standard (cochain) complex of the Courant-Dorfman algebra
$\mathcal{E}$, the resulting cohomology is called the standard cohomology
of $\mathcal{E}$, and denoted by $H_{st}^{\bullet}(\mathcal{E})$.
\end{defn}

\section{$H$-Standard cohomology}

In this section, we give the definition of $H$-standard cohomology
for Courant-Dorfman algebras and Leibniz algebras respectively. The
inspiration comes from Roytenberg's construction of standard cohomology
for Courant-Dorfman algebras \cite{RoytCDA}.

\subsection{For Courant-Dorfman algebras}

Given a Courant-Dorfman algebra $(\mathcal{E},R,\langle\cdot,\cdot\rangle,\partial,\circ)$,
suppose $\mathcal{H}$ is an $R$-submodule as well as an isotropic
ideal of $\mathcal{E}$ containing $\partial R$, and $(\mathcal{V},\nabla)$
is an $\mathcal{H}$-representation of $\mathcal{E}$, which is defined
as follows: 
\begin{defn}
An $\mathcal{H}$-trivial representation (or $\mathcal{H}$-representation
for short) of a Courant-Dorfman algebra $(\mathcal{E},R,\langle\cdot,\cdot\rangle,\partial,\circ)$
is a pair $(\mathcal{V},\nabla)$, where $\mathcal{V}$ is an $R$-module,
and $\nabla:\mathcal{E}\rightarrow Der(\mathcal{V})$ is a homomorphism
of Leibniz algebras such that:
\begin{eqnarray*}
\nabla_{\alpha}v & = & 0\\
\nabla_{fe}v & = & f\nabla_{e}v\\
\nabla_{e}(fv) & = & (\rho(e)f)v+f(\nabla_{e}v)\qquad\forall\alpha\in\mathcal{H},e\in\mathcal{E},f\in R,v\in\mathcal{V}.
\end{eqnarray*}
\end{defn}
\begin{example}
Since 
\[
\rho(\alpha)f=\langle\alpha,\partial f\rangle=0,\qquad\forall\alpha\in\mathcal{H},f\in R,
\]
$\mathcal{H}$ is actually an $R$-submodule of $ker\rho$. It is
easily checked that $(R,\rho)$ is an $\mathcal{H}$-representation
of $\mathcal{E}$.
\end{example}
In the following of this subsection, we always assume
\[
e\in\mathcal{E},\ \alpha\in\mathcal{H},\ f\in R.
\]

Denote by $C^{n}(\mathcal{E},\mathcal{H},\mathcal{V})$ the space
of all sequences
\[
\omega=(\omega_{0},\cdots,\omega_{[\frac{n}{2}]}),
\]
where the $\mathfrak{k}$-linear maps
\[
\omega_{k}:\ (\otimes^{n-2k}\mathcal{E})\otimes(\odot^{k}\mathcal{H})\rightarrow\mathcal{V}
\]
satisfy the following conditions:

1). Weak skew-symmetricity in arguments of $\mathcal{E}$.

$\forall k$, $\omega_{k}$ is weakly skew-symmetric up to $\omega_{k+1}$:
\begin{eqnarray*}
 &  & \omega_{k}(e_{1},\cdots e_{a},e_{a+1},\cdots e_{n-2k};\alpha_{1},\cdots\alpha_{k})+\omega_{k}(e_{1},\cdots e_{a+1},e_{a},\cdots e_{n-2k};\alpha_{1},\cdots\alpha_{k})\\
 & = & -\omega_{k+1}(e_{1},\cdots\widehat{e_{a}},\widehat{e_{a+1}},\cdots e_{n-2k};\partial\langle e_{a},e_{b}\rangle,\alpha_{1},\cdots\alpha_{k}),
\end{eqnarray*}

2). Weak $R$-linearity in arguments of $\mathcal{E}$.

$\forall k$, $\omega_{k}$ is weakly $R$-linear up to $\omega_{k+1}$:
\begin{eqnarray*}
 &  & \omega_{k}(e_{1},\cdots fe_{a},\cdots e_{n-2k};\alpha_{1},\cdots\alpha_{k})\\
 & = & f\omega_{k}(e_{1},\cdots e_{a},\cdots e_{n-2k};\cdots)+\sum_{b>a}(-1)^{b-a}\langle e_{a},e_{b}\rangle\omega_{k+1}(e_{1},\cdots\widehat{e_{a}},\cdots\widehat{e_{b}},\cdots e_{n-2k};\partial f,\cdots),
\end{eqnarray*}

3). $R$-linearity in arguments of $\mathcal{H}$.

$\forall k$, $\omega_{k}$ is $R$-linear in arguments of $\mathcal{H}$:
\[
\omega_{k}(e_{1},\cdots e_{n-2k};\alpha_{1},\cdots f\alpha_{l},\cdots\alpha_{k})=f\omega_{k}(e_{1},\cdots e_{n-2k};\alpha_{1},\cdots\alpha_{l},\cdots\alpha_{k}).
\]

\begin{thm}
\label{thm:standard complex CDA}$C^{\bullet}(\mathcal{E},\mathcal{H},\mathcal{V})\triangleq\bigoplus_{n}C^{n}(\mathcal{E},\mathcal{H},\mathcal{V})$
is a cochain complex under the coboundary map $d=d_{0}+\delta+d^{\prime}$,
where $d_{0}$ is the coboundary map (Equation \ref{eq: Leibniz coboundary})
corresponding to the Leibniz cohomology of $\mathcal{E}$ with coefficients
in $(\mathcal{V},\nabla,-\nabla)$, and $\delta,d^{\prime}$ are defined
for any $\omega\in C^{n}(\mathcal{E},\mathcal{H},\mathcal{V})$ respectively
by: 
\begin{eqnarray*}
(\delta\omega)_{k}(e_{1},\cdots e_{n+1-2k};\alpha_{1},\cdots\alpha_{k}) & \triangleq & \sum_{i}\omega_{k-1}(\alpha_{i},e_{1},\cdots,e_{n+1-2k};\cdots\widehat{\alpha_{i}},\cdots),\\
(d^{\prime}\omega)_{k}(e_{1},\cdots e_{n+1-2k};\alpha_{1},\cdots\alpha_{k}) & \triangleq & \sum_{a,i}(-1)^{a+1}\omega_{k}(\cdots\widehat{e_{a}}\cdots;\cdots\widehat{\alpha_{i}},\alpha_{i}\circ e_{a}\cdots).
\end{eqnarray*}
\end{thm}
\begin{lem}
$C^{\bullet}(\mathcal{E},\mathcal{H},\mathcal{V})$ is closed under
$d=d_{0}+\delta+d^{\prime}$.\end{lem}
\begin{proof}
$\forall\omega\in C^{n}(\mathcal{E},\mathcal{H},\mathcal{V})$, we
need to prove that $d\omega\in C^{n+1}(\mathcal{E},\mathcal{H},\mathcal{V})$.

First, we prove the weak skew-symmetricity in arguments of $\mathcal{E}$.
We will calculate $d_{0},\delta,d^{\prime}$ parts separately. The
calculations are straightforward from the definitions but rather tedious.
To save space, we omit the details, and only list the results of calculations
here.
\begin{eqnarray}
 &  & (d_{0}\omega)_{k}(e_{1},\cdots e_{i},e_{i+1}\cdots e_{n+1-2k};\alpha_{1},\cdots\alpha_{k})+(d_{0}\omega)_{k}(e_{1},\cdots e_{i+1},e_{i}\cdots;\cdots)\label{eq:d0 weak skewsym}\\
 & = & -(d_{0}\omega)_{k+1}(e_{1},\cdots\widehat{e_{i}},\widehat{e_{i+1}},\cdots;\partial\langle e_{i},e_{i+1}\rangle,\cdots)-\omega_{k}(\partial\langle e_{i},e_{i+1}\rangle,e_{1},\cdots,\widehat{e_{i}},\widehat{e_{i+1}},\cdots;\cdots),\nonumber 
\end{eqnarray}
\begin{eqnarray*}
 &  & (\delta\omega)_{k}(e_{1},\cdots e_{i},e_{i+1}\cdots e_{n+1-2k};\alpha_{1},\cdots\alpha_{k})+(\delta\omega)_{k}(e_{1},\cdots e_{i+1},e_{i}\cdots;\cdots)\\
 & = & -(\delta\omega)_{k+1}(e_{1},\cdots,\widehat{e_{i}},\widehat{e_{i+1}},\cdots;\partial\langle e_{i},e_{i+1}\rangle,\cdots)+\omega_{k}(\partial\langle e_{i},e_{i+1}\rangle,e_{1},\cdots,\widehat{e_{i}},\widehat{e_{i+1}},\cdots;\cdots),
\end{eqnarray*}
\begin{eqnarray}
 &  & (d^{\prime}\omega)_{k}(e_{1}\cdots e_{i},e_{i+1}\cdots e_{n+1-2k};\alpha_{1},\cdots\alpha_{k})+(d^{\prime}\omega)_{k}(e_{1},\cdots e_{i+1},e_{i}\cdots;\cdots)\label{eq:d prime weak skewsym}\\
 & = & -(d^{\prime}\omega)_{k+1}(e_{1},\cdots,\widehat{e_{i}},\widehat{e_{i+1}},\cdots;\partial\langle e_{i},e_{i+1}\rangle,\cdots).\nonumber 
\end{eqnarray}

The sum of the three equations above tells:
\begin{eqnarray*}
 &  & (d\omega)_{k}(e_{1},\cdots e_{i},e_{i+1}\cdots e_{n+1-2k};\alpha_{1},\cdots\alpha_{k})+(d\omega)_{k}(e_{1},\cdots e_{i+1},e_{i}\cdots;\cdots)\\
 & = & -(d\omega)_{k+1}(e_{1},\cdots,\widehat{e_{i}},\widehat{e_{i+1}},\cdots;\partial\langle e_{i},e_{i+1}\rangle,\cdots),
\end{eqnarray*}
i.e. $(d\omega)_{k}$ is weakly skew-symmetric up to $(d\omega)_{k+1}$.

Next, we prove the weak $R$-linearity in arguments of $\mathcal{E}$.
By direct calculations, we have the following:
\begin{eqnarray}
 &  & (d_{0}\omega)_{k}(e_{1},\cdots fe_{i},\cdots e_{n+1-2k};\alpha_{1},\cdots\alpha_{k})-f(d_{0}\omega)_{k}(e_{1},\cdots e_{i},\cdots;\cdots)\label{eq:d0 weak R linear}\\
 & = & \sum_{a>i}(-1)^{a-i}\langle e_{i},e_{a}\rangle(d_{0}\omega)_{k+1}(\cdots\widehat{e_{i}},\cdots\widehat{e_{a}},\cdots;\partial f,\cdots)\nonumber \\
 &  & +\sum_{i<a}(-1)^{a-i}\langle e_{i},e_{a}\rangle\omega_{k}(\partial f,e_{1},\cdots\widehat{e_{i}},\cdots\widehat{e_{a}},\cdots;\cdots)\nonumber 
\end{eqnarray}

\begin{eqnarray*}
 &  & (\delta\omega)_{k}(e_{1},\cdots fe_{i},\cdots e_{n+1-2k};\alpha_{1},\cdots\alpha_{k})-f(\delta\omega)_{k}(e_{1},\cdots e_{i},\cdots;\cdots)\\
 & = & \sum_{i<a}(-1)^{a-i}\langle e_{i},e_{a}\rangle(\delta\omega)_{k+1}(e_{1},\cdots\widehat{e_{i}},\cdots\widehat{e_{a}},\cdots;\partial f,\cdots)\\
 &  & -\sum_{i<a}(-1)^{a-i}\langle e_{i},e_{a}\rangle\omega_{k}(\partial f,e_{1},\cdots\widehat{e_{i}},\cdots\widehat{e_{a}},\cdots;\cdots)
\end{eqnarray*}

\begin{eqnarray}
 &  & (d^{\prime}\omega)_{k}(e_{1},\cdots fe_{i},\cdots e_{n+1-2k};\alpha_{1},\cdots\alpha_{k})-f(d^{\prime}\omega)_{k}(e_{1},\cdots e_{i},\cdots;\cdots)\label{eq:d prime weak R linear}\\
 & = & \sum_{a>i}(-1)^{a-i}\langle e_{i},e_{a}\rangle(d^{\prime}\omega)_{k+1}(e_{1},\cdots\widehat{e_{i}},\cdots\widehat{e_{a}},\cdots;\partial f,\cdots)\nonumber 
\end{eqnarray}

The sum of the three equations above is:
\begin{eqnarray*}
 &  & (d\omega)_{k}(e_{1},\cdots fe_{i},\cdots e_{n+1-2k};\alpha_{1},\cdots\alpha_{k})-f(d\omega)_{k}(e_{1},\cdots e_{i},\cdots;\cdots)\\
 & = & \sum_{a>i}(-1)^{a-i}\langle e_{i},e_{a}\rangle(d\omega)_{k+1}(e_{1},\cdots\widehat{e_{i}},\cdots\widehat{e_{a}},\cdots;\partial f,\cdots)
\end{eqnarray*}
i.e. $(d\omega)_{k}$ is weakly $R$-linear up to $(d\omega)_{k+1}$.

Finally, we need to prove the $R$-linearity in arguments of $\mathcal{H}$.
Again by direct calculations, we have the following:
\begin{eqnarray}
 &  & (d_{0}\omega+d^{\prime}\omega)_{k}(e_{1},\cdots e_{n+1-2k};\alpha_{1},\cdots f\alpha_{l},\cdots\alpha_{k})-f(d_{0}\omega+d^{\prime}\omega)_{k}(\cdots;\cdots\alpha_{l},\cdots)\label{eq:d R linear}\\
 & = & \sum_{a}(-1)^{a+1}\langle\alpha_{l},e_{a}\rangle\omega_{k}(\cdots\widehat{e_{a}},\cdots;\partial f,\cdots\widehat{\alpha_{l}},\cdots),\nonumber 
\end{eqnarray}

and 
\begin{eqnarray*}
 &  & (\delta\omega)_{k}(e_{1},\cdots e_{n+1-2k};\alpha_{1},\cdots f\alpha_{l},\cdots\alpha_{k})-f(\delta\omega)_{k}(\cdots;\cdots\alpha_{l},\cdots)\\
 & = & \sum_{a}(-1)^{a}\langle\alpha_{l},e_{a}\rangle\omega_{k}(\cdots\widehat{e_{a}},\cdots;\partial f,\cdots\widehat{\alpha_{l}},\cdots),
\end{eqnarray*}

so
\[
(d\omega)_{k}(e_{1},\cdots e_{n+1-2k};\alpha_{1},\cdots f\alpha_{l},\cdots\alpha_{k})-f(d\omega)_{k}(\cdots;\cdots\alpha_{l},\cdots)=0.
\]

The lemma is proved.
\end{proof}
Now we turn to the proof of Theorem \ref{thm:standard complex CDA}:
\begin{proof}
By the lemma above, we only need to prove $d^{2}=0$.

$d^{2}$ can be divided into six parts 
\[
d^{2}=d_{0}^{2}+\delta^{2}+(d_{0}\circ\delta+\delta\circ d_{0})+(d_{0}\circ d^{\prime}+d^{\prime}\circ d_{0})+(\delta\circ d^{\prime}+d^{\prime}\circ\delta)+d^{\prime2}.
\]
The first part equals $0$, so we only need to compute the other five
parts. By direct calculations, we have the following:

\[
(\delta^{2}\omega)_{k}(e_{1},\cdots,e_{n+2-2k};\alpha_{1},\cdots,\alpha_{k})=0,
\]
\[
((d_{0}\circ\delta+\delta\circ d_{0})\omega)_{k}(e_{1},\cdots;\alpha_{1},\cdots)=-\sum_{i,a}\omega_{k-1}(\cdots\widehat{e_{a}},\alpha_{i}\circ e_{a},\cdots;\cdots\widehat{\alpha_{i}},\cdots),
\]
\[
((d_{0}\circ d^{\prime}+d^{\prime}\circ d_{0})\omega)_{k}(e_{1},\cdots;\alpha_{1},\cdots)=\sum_{j,a<c}(-1)^{a+c}\omega_{k}(\cdots\widehat{e_{a}},\cdots\widehat{e_{c}},\cdots;\cdots\alpha_{j}\circ(e_{a}\circ e_{c}),\cdots),
\]
\[
((\delta\circ d^{\prime}+d^{\prime}\circ\delta)\omega)_{k}(e_{1},\cdots;\alpha_{1},\cdots)=\sum_{j,a}(-1)^{a+1}\omega_{k-1}(\alpha_{j}\circ e_{a},\cdots\widehat{e_{a}},\cdots;\cdots\widehat{\alpha_{j}},\cdots),
\]
\[
(d^{\prime2}\omega)_{k}(e_{1},\cdots;\alpha_{1},\cdots)=\sum_{a<b,i}(-1)^{a+b}\omega_{k}(\cdots\widehat{e_{a}},\cdots\widehat{e_{b}},\cdots;\cdots\widehat{\alpha_{i}},(\alpha_{i}\circ e_{b})\circ e_{a}-(\alpha_{i}\circ e_{a})\circ e_{b},\cdots),
\]

The sum of the above equations is:
\begin{eqnarray*}
 &  & (d^{2}\omega)_{k}(e_{1},\cdots e_{n+2-2k};\alpha_{1},\cdots\alpha_{k})\\
 & = & \sum_{j,a}(-1)^{a+1}\omega_{k-1}(\alpha_{j}\circ e_{a},\cdots\widehat{e_{a}},\cdots;\cdots\widehat{\alpha_{j}},\cdots)-\sum_{i,a}\omega_{k-1}(\cdots\alpha_{i}\circ e_{a},\cdots;\cdots\widehat{\alpha_{i}},\cdots)\\
 &  & +\sum_{j,a<b}(-1)^{a+b}\omega_{k}(\cdots\widehat{e_{a}},\cdots\widehat{e_{b}},\cdots;\cdots\alpha_{j}\circ(e_{a}\circ e_{b}),\cdots)\\
 &  & +\sum_{a<b,i}(-1)^{a+b}\omega_{k}(\cdots\widehat{e_{a}},\cdots\widehat{e_{b}},\cdots;\cdots\widehat{\alpha_{i}},(\alpha_{i}\circ e_{b})\circ e_{a}-(\alpha_{i}\circ e_{a})\circ e_{b},\cdots)\\
 & = & \sum_{i,b<a}(-1)^{a+b}\big(\omega_{k-1}(\cdots\alpha_{i}\circ e_{a},e_{b},\cdots\widehat{e_{a}},\cdots;\cdots\widehat{\alpha_{i}},\cdots)\\
 &  & \quad+\omega_{k-1}(\cdots e_{b},\alpha_{i}\circ e_{a},\cdots\widehat{e_{a}},\cdots;\cdots\widehat{\alpha_{i}},\cdots)\big)\\
 &  & +\sum_{i,a<b}(-1)^{a+b}\omega_{k}(\cdots\widehat{e_{a}},\cdots\widehat{e_{b}},\cdots;\cdots\langle e_{a},\alpha_{i}\circ e_{b}\rangle,\cdots)\\
 & = & 0
\end{eqnarray*}

The proof is finished.\end{proof}
\begin{defn}
\label{Def:generalized standard cohomology CDA}With the notations
above, $(C^{\bullet}(\mathcal{E},\mathcal{H},\mathcal{V}),d)$ is
called $\mathcal{H}$-standard complex of $\mathcal{E}$ with coefficients
in $\mathcal{V}$. The resulting cohomology, denoted by $H^{\bullet}(\mathcal{E},\mathcal{H},\mathcal{V})$,
is called $\mathcal{H}$-standard cohomology of $\mathcal{E}$ with
coefficients in $\mathcal{V}$.
\end{defn}
Let's consider the standard cohomology in lower degrees:

Degree 0:

$H^{0}(\mathcal{E},\mathcal{H},\mathcal{V})$ is the submodule of
$\mathcal{V}$ consisting of all invariants, i.e.
\[
H^{0}(\mathcal{E},\mathcal{H},\mathcal{V})=\{v\in\mathcal{V}|\nabla_{e}v=0,\ \forall e\in\mathcal{E}\}.
\]

Degree 1:

A cocycle $\omega$ in $C^{1}(\mathcal{E},\mathcal{H},\mathcal{V})$
is a map $\omega_{0}:\mathcal{E}\rightarrow\mathcal{V}$ satisfying:
\[
\omega_{0}(e_{1}\circ e_{2})=\nabla_{e_{1}}\omega_{0}(e_{2})-\nabla_{e_{2}}\omega_{0}(e_{1}),\quad\forall e_{1},e_{2}\in\mathcal{E}
\]
and
\[
\omega_{0}(\alpha)=0,\quad\forall\alpha\in\mathcal{H}.
\]

The first equation above tells that $\omega_{0}$ is a derivation
from $\mathcal{E}$ to $\mathcal{V}$, while the second equation tells
that $\omega_{0}$ induces a map from $\mathcal{E}/\mathcal{H}$ to
$\mathcal{V}$.

$\eta\in C^{1}(\mathcal{E},\mathcal{H},\mathcal{V})$ is a coboundary
iff there exists $v\in\mathcal{V}$ such that:
\[
\eta_{0}(e)=\nabla_{e}v,\quad\forall e\in\mathcal{E},
\]
i.e. $\eta_{0}$ is an inner derivation from $\mathcal{E}$ to $\mathcal{V}$.

Thus $H^{1}(\mathcal{E},\mathcal{H},\mathcal{V})$ is the space of
``outer derivations'': \{derivations\}/\{inner derivations\} from
$\mathcal{E}$ to $\mathcal{V}$ which act trivially on $\mathcal{H}$.
Or equivalently, $H^{1}(\mathcal{E},\mathcal{H},\mathcal{V})$ is
the space of outer derivations from $\mathcal{E}/\mathcal{H}$ to
$\mathcal{V}$.

Degree 2:

$\omega=(\omega_{0},\omega_{1})\in C^{2}(\mathcal{E},\mathcal{H},\mathcal{V})$
is a $2$-cocycle iff:
\begin{eqnarray}
\nabla_{e_{1}}\omega_{0}(e_{2},e_{3})-\nabla_{e_{2}}\omega_{0}(e_{1},e_{3})+\nabla_{e_{3}}\omega_{0}(e_{1},e_{2})\nonumber \\
-\omega_{0}(e_{1}\circ e_{2},e_{3})-\omega_{0}(e_{2},e_{1}\circ e_{3})+\omega_{0}(e_{1},e_{2}\circ e_{3}) & = & 0\label{eq:2cocycle0}
\end{eqnarray}
and
\begin{equation}
\nabla_{e}\omega_{1}(\alpha)+\omega_{0}(\alpha,e)+\omega_{1}(\alpha\circ e)=0\label{eq:2cocycle1}
\end{equation}
$\forall e,e_{1},e_{2},e_{3}\in\mathcal{E},\alpha\in\mathcal{H}.$

Equation \ref{eq:2cocycle0} holds iff the bracket on $\bar{\mathcal{E}}\triangleq\mathcal{E}\oplus\mathcal{V}$
defined for any $e_{1},e_{2}\in\mathcal{E}$, $v_{1},v_{2}\in\mathcal{V}$
by:
\[
(e_{1}+v_{1})\bar{\circ}(e_{2}+v_{2})\triangleq e_{1}\circ e_{2}+\big(\nabla_{e_{1}}v_{2}-\nabla_{e_{2}}v_{1}+\omega_{0}(e_{1},e_{2})\big)
\]
is a Leibniz bracket. Furthermore, if Equation \ref{eq:2cocycle1}
also holds, it is easily checked that $(\bar{\mathcal{E}},R,\overline{\langle\cdot,\cdot\rangle},\bar{\partial},\bar{\circ})$
is a Courant-Dorfman algebra, where $\overline{\langle\cdot,\cdot\rangle}$
and $\bar{\partial}$ are defined as:
\begin{eqnarray*}
\overline{\langle e_{1}+v_{1},e_{2}+v_{2}\rangle} & = & \langle e_{1},e_{2}\rangle\\
\bar{\partial}f & = & \partial f-\omega_{1}(\partial f).
\end{eqnarray*}
Actually Equation \ref{eq:2cocycle0} implies that
\[
\bar{\mathcal{H}}\triangleq\{\alpha-\omega_{1}(\alpha)|\alpha\in\mathcal{H}\}
\]
is an ideal of $\bar{\mathcal{E}}$.

In a summation, $2$-cocycles are in 1-1 correspondence with central
extensions of Courant-Dorfman algebras which are split as metric $R$-modules:
\[
0\rightarrow\mathcal{V}\rightarrow\bar{\mathcal{E}}\rightarrow\mathcal{E}\rightarrow0
\]
such that $\bar{\mathcal{H}}$ is an ideal of $\bar{\mathcal{E}}$.

The central extensions determined by $2$-cocycles $\omega_{1},\omega_{2}$
are isomorphic iff $\omega_{1}-\omega_{2}=d\lambda,$ for some $\lambda\in C^{1}(\mathcal{E},\mathcal{H},\mathcal{V})$.

Thus $H^{2}(\mathcal{E},\mathcal{H},\mathcal{V})$ classifies isomorphism
classes of central extensions of Courant-Dorfman algebras which are
split as metric $R$-modules:
\[
0\rightarrow\mathcal{V}\rightarrow\bar{\mathcal{E}}\rightarrow\mathcal{E}\rightarrow0
\]
such that $\bar{\mathcal{H}}$ is an ideal of $\bar{\mathcal{E}}$.

\subsection{For Leibniz algebras}

Assume $L$ is a Leibniz algebra with left center $Z$. There is a
symmetric bilinear product $(\cdot,\cdot):L\otimes L\rightarrow Z$
defined as: 

\[
(e_{1},e_{2})\triangleq e_{1}\circ e_{2}+e_{2}\circ e_{1},\qquad\forall e_{1},e_{2}\in L.
\]

It is easily checked that such defined bilinear product is invariant,
i.e.
\[
\rho(e_{1})(e_{2},e_{3})=(e_{1}\circ e_{2},e_{3})+(e_{2},e_{1}\circ e_{3}),\qquad\forall e_{1},e_{2},e_{3}\in L.
\]

Let $H\supseteq Z$ be an isotropic ideal in $L$. Let $(V,\tau)$
be an $H$-representation of $L$, which is defined as follows:
\begin{defn}
An $H$-trivial left representation (or $H$-representation for short)
of a Leibniz algebra $L$ is a pair $(V,\tau)$, where $V$ is vector
space, and $\tau:L\rightarrow gl(V)$ is a homomorphism of Leibniz
algebras such that:
\[
\tau(h)v=0,\quad\forall h\in H,v\in V.
\]
\end{defn}
\begin{example}
Since 
\[
\rho(h)z=h\circ z=(h,z)=0,\qquad\forall h\in H,z\in Z,
\]
 $(Z,\rho)$ is an $H$-representation.
\end{example}
Denote by $C^{n}(L,H,V)$ the space of all sequences $\omega=(\omega_{0},\cdots,\omega_{[\frac{n}{2}]})$,
where $\omega_{k}$ is a linear map from $(\otimes^{n-2k}L)\otimes(\odot^{k}H)$
to $V$, $\forall k$, and is weakly skew-symmetric in arguments of
$L$ up to $\omega_{k+1}$:
\begin{eqnarray*}
 &  & \omega_{k}(e_{1},\cdots e_{i},e_{i+1,}\cdots e_{n-2k};h_{1},\cdots h_{k})+\omega_{k}(e_{1},\cdots e_{i+1},e_{i,}\cdots e_{n-2k};h_{1},\cdots h_{k})\\
 & = & -\omega_{k+1}(\cdots\widehat{e_{i}},\widehat{e_{i+1}},\cdots;(e_{i},e_{i+1}),\cdots)
\end{eqnarray*}
$\forall e_{j}\in L,h_{l}\in H$.
\begin{thm}
\label{thm: standard complex}$C^{\bullet}(L,H,V)\triangleq\bigoplus_{n}C^{n}(L,H,V)$
is a cochain complex, under the coboundary map $d=d_{0}+\delta+d^{\prime}$,
where $d_{0}$ is the coboundary map (Equation \ref{eq: Leibniz coboundary})
corresponding to the Leibniz cohomology of $L$ with coefficients
in $(V,\tau,-\tau)$, and $\delta,d^{\prime}$ are defined for any
$\omega\in C^{n}(L,H,V)$ respectively by:

\begin{eqnarray*}
(\delta\omega)_{k}(e_{1},\cdots,e_{n+1-2k};h_{1},\cdots h_{k}) & \triangleq & \sum_{j}\omega_{k-1}(\alpha_{j},e_{1},\cdots,e_{n+1-2k};h_{1},\cdots\widehat{h_{j}},\cdots h_{k})\\
(d^{\prime}\omega)_{k}(e_{1},\cdots e_{n+1-2k};h_{1},\cdots h_{k}) & \triangleq & \sum_{a,j}(-1)^{a+1}\omega_{k}(\cdots\widehat{e_{a}},\cdots;\cdots\widehat{h_{j}},h_{j}\circ e_{a},\cdots)
\end{eqnarray*}
$\forall e_{a}\in L,h_{i}\in H.$ ($(\delta\omega)_{0}$ is defined
to be $0$.)\end{thm}
\begin{proof}
The proof of this theorem is quite similar to that of Theorem \ref{thm:standard complex CDA},
so we omit it here.\end{proof}
\begin{defn}
\label{Def:standard cohomology Leibniz} $(C^{\bullet}(L,H,V),d)$
is called the $H$-standard complex of $L$ with coefficients in $V$.
The resulting cohomology, denoted by $H^{\bullet}(L,H,V)$ is called
the $H$-standard cohomology of $L$ with coefficients in $V$.
\end{defn}
The $H$-standard cohomology of $L$ in degree $0,1,2$ have similar
interpretations to the case of Courant-Dorfman algebras:

$H^{0}(L,H,V)$ is the submodule of $V$ consisting of all invariants.

$H^{1}(L,H,V)$ is the space of outer derivations from $L$ to $V$
acting trivially on $H$.

$H^{2}(L,H,V)$ classfies the equivalence classes of abelian extensions
of $L$ by $V$:
\[
0\rightarrow V\rightarrow\bar{L}\rightarrow L\rightarrow0
\]

such that $\bar{H}$ is an ideal of $\bar{L}$.

\section{Isomorphism theorem for Courant-Dorfman algebra}

In this section, we present one of our main results in this paper,
which is a generalization of a theorem of Ginot-Grutzmann (conjectured
by Stienon-Xu) for transitive Courant algebroids.

Let $\mathcal{E},\mathcal{H},\mathcal{V}$ be as described in the
last section. Since $\mathcal{H}$ is an ideal in $\mathcal{E}$ containing
$\partial R$, it is easily checked that $\mathcal{E}/\mathcal{H}$
is a Lie-Rinehart algebra with the induced anchor map(still denoted
by $\rho$):
\[
\rho([e])f\triangleq\rho(e)f,\qquad\forall e\in\mathcal{E},f\in R,
\]
and induced bracket:
\[
[e_{1}]\circ[e_{2}]\triangleq[e_{1}\circ e_{2}],\qquad\forall e_{1},e_{2}\in\mathcal{E}.
\]
Moreover, $\mathcal{V}$ becomes a representation of $\mathcal{E}/\mathcal{H}$
with the induced action (still denoted by $\tau$):
\[
\tau([e])v\triangleq\tau(e)v,\qquad\forall e\in\mathcal{E},v\in\mathcal{V}.
\]
As a result, we have the Chevalley-Eilenberg complex $(C_{CE}^{\bullet}(\mathcal{E}/\mathcal{H},\mathcal{V}),d_{CE})$
of $\mathcal{E}/\mathcal{H}$ with coefficients in $\mathcal{V}$,
and the corresponding cohomology $H_{CE}^{\bullet}(\mathcal{E}/\mathcal{H},\mathcal{V})$.
\begin{thm}
\label{thm:isomorphism CDA} Given a Courant-Dorfman algebra $(\mathcal{E},R,\langle\cdot,\cdot\rangle,\partial,\circ)$,
an $R$-submodule $\mathcal{H}\supseteq\partial R$ which is an isotropic
ideal of $\mathcal{E}$, and an $\mathcal{H}$-representation $(\mathcal{V},\nabla)$.
If the quotient module $\mathcal{E}/\mathcal{H}$ is projective, then
we have:
\[
H^{\bullet}(\mathcal{E},\mathcal{H},\mathcal{V})\cong H_{CE}^{\bullet}(\mathcal{E}/\mathcal{H},\mathcal{V}).
\]

\end{thm}
Before proof of this theorem, we prove the following two lemmas first.
\begin{lem}
\label{lem: CE nv}$(C_{CE}^{\bullet}(\mathcal{E}/\mathcal{H},\mathcal{V}),d_{CE})$
is isomorphic to the following subcomplex of $(C^{\bullet}(\mathcal{E},\mathcal{H},\mathcal{V}),d)$:
\[
C_{nv}^{\bullet}(\mathcal{E},\mathcal{H},\mathcal{V})\triangleq\{\omega\in C^{\bullet}(\mathcal{E},\mathcal{H},\mathcal{V})|\omega_{k}=0,\ \forall k\geq1,\ \iota_{\alpha}\omega_{0}=0,\ \forall\alpha\in\mathcal{H}\}
\]
\end{lem}
\begin{proof}
Obviously $C_{nv}^{\bullet}(\mathcal{E},\mathcal{H},\mathcal{V})$
is a subcomplex of $C^{\bullet}(\mathcal{E},\mathcal{H},\mathcal{V})$.

And it is easily checked that the following two maps $\varphi,\phi$
are well-defined cochain maps and invertible to each other:
\begin{eqnarray*}
\varphi:C_{nv}^{\bullet}(\mathcal{E},\mathcal{H},\mathcal{V}) & \rightarrow & C_{CE}^{\bullet}(\mathcal{E}/\mathcal{H},\mathcal{V})\\
\varphi(\eta)([e_{1}],\cdots[e_{n}]) & \triangleq & \eta(e_{1},\cdots e_{n})\quad\forall\eta\in C_{nv}^{n}(\mathcal{E},\mathcal{H},\mathcal{V}),e_{a}\in\mathcal{E},
\end{eqnarray*}
and
\begin{eqnarray*}
\phi:C_{CE}^{\bullet}(\mathcal{E}/\mathcal{H},\mathcal{V}) & \rightarrow & C_{nv}^{\bullet}(\mathcal{E},\mathcal{H},\mathcal{V})\\
\phi(\zeta)(e_{1},\cdots e_{n}) & \triangleq & \zeta([e_{1}],\cdots[e_{n}])\quad\forall\zeta\in C_{CE}^{\bullet}(\mathcal{E}/\mathcal{H},\mathcal{V}),e_{a}\in\mathcal{E}.
\end{eqnarray*}
\end{proof}
\begin{lem}
\label{lem:lambda}Given any $\omega\in C^{n}(\mathcal{E},\mathcal{H},\mathcal{V})$,
if $(d\omega)_{k}=0,\ \forall k\geq1$, then there exists $\eta\in C_{nv}^{n}(\mathcal{E},\mathcal{H},\mathcal{V})$
and $\lambda\in C^{n-1}(\mathcal{E},\mathcal{H},\mathcal{V})$ such
that $\omega=\eta+d\lambda$.\end{lem}
\begin{proof}
Since the quotient $\mathcal{E}/\mathcal{H}$ is a projective module,
there exists an $R$-module decomposition: $\mathcal{E}=\mathcal{H}\oplus\mathcal{X}$.
We will give an inductive construction of $\lambda$ and $\beta$.
The construction depends on the decomposition, but the cohomology
class of $\beta$ doesn't depend on the decomposition. 

Suppose $n=2m\ or\ 2m-1$, we will define $\lambda_{m-1},\lambda_{m-2},\cdots,\lambda_{0}$
one by one, so that each $\lambda_{p}:\ \otimes^{n-1-2p}\mathcal{E}\otimes\odot^{p}\mathcal{H}\rightarrow\mathcal{V},\ 0\leq p\leq m-1$
satisfies the following conditions, which we call ``Lambda Conditions'':

1). $\lambda_{p}$ is weakly skew-symmetric in arguments of $\mathcal{E}$
up to $\lambda_{p+1}$,

2). $\lambda_{p}$ is weakly $R$-linear in arguments of $\mathcal{E}$
up to $\lambda_{p+1}$,

3). $\lambda_{p}$ is $R$-linear in arguments of $R$,

4). $\omega_{p+1}=(d\lambda)_{p+1}$,

5). $\sum_{i}(\omega_{p}-d_{0}\lambda_{p}-d^{\prime}\lambda_{p})(\alpha_{i},e_{1},\cdots e_{n-1-2p};\alpha_{1},\cdots\widehat{\alpha_{i}},\cdots\alpha_{p+1})=0$,
$\forall\alpha_{j}\in\mathcal{H},e_{a}\in\mathcal{E}$.

The construction of $\lambda_{m-1},\lambda_{m-2},\cdots,\lambda_{0}$
is done in the following four steps.

Step 1:

Construction of $\lambda_{m-1}$:

When $n=2m-1$ is odd, let 
\[
\lambda_{m-1}(\alpha_{1},\cdots\alpha_{m-1})=0,\qquad\forall\alpha_{j}\in\mathcal{H}.
\]

When $n=2m$ is even, let 
\[
\lambda_{m-1}(\beta;\alpha_{1},\cdots\alpha_{m-1})=\frac{1}{m}\omega_{m}(\beta,\alpha_{1},\cdots\alpha_{m-1}),\ \forall\beta,\alpha_{j}\in\mathcal{H}
\]
 and 
\[
\lambda_{m-1}(x;\alpha_{1},\cdots\alpha_{m-1})=0,\ \forall x\in\mathcal{X},\alpha_{j}\in\mathcal{H}.
\]

It is obvious that $\lambda_{m-1}$ defined above satisfies Lambda
Conditions 1) - 4). So we only need to prove Lambda Condition 5):

When $n=2m-1$, 
\[
\sum_{i}(\omega_{m-1}-d_{0}\lambda_{m-1}-d^{\prime}\lambda_{m-1})(\alpha_{i};\alpha_{1},\cdots\widehat{\alpha_{i}},\cdots\alpha_{m})=(d\omega)_{m}(\alpha_{1},\cdots\alpha_{m})=0.
\]

When $n=2m$, the left hand side in condition 3) equals
\begin{eqnarray*}
 &  & \sum_{i}(\omega_{m-1}-d_{0}\lambda_{m-1}-d^{\prime}\lambda_{m-1})(\alpha_{i},e;\alpha_{1},\cdots\widehat{\alpha_{i}},\cdots\alpha_{m})\\
 & = & (\delta\omega)_{m}(e;\alpha_{1},\cdots,\alpha_{m})+\sum_{i}\nabla_{e}\lambda_{m-1}(\alpha_{i};\cdots\widehat{\alpha_{i}},\cdots)+\sum_{i}\lambda_{m-1}(\alpha_{i}\circ e;\cdots\widehat{\alpha_{i}},\cdots)\\
 &  & +\sum_{j\neq i}(-1)\lambda_{m-1}(e;\cdots\widehat{\alpha_{i}},\cdots\alpha_{j}\circ\alpha_{i},\cdots)+\sum_{j\neq i}\lambda_{m-1}(\alpha_{i};\cdots\widehat{\alpha_{i}},\cdots,\alpha_{j}\circ e,\cdots)\\
 & = & (\delta\omega)_{m}(e;\alpha_{1},\cdots\alpha_{m})+\nabla_{e}\omega_{m}(\alpha_{1},\cdots\alpha_{m})+\frac{1}{m}\sum_{i}\omega_{m}(\alpha_{1},\cdots\alpha_{i}\circ e,\cdots\alpha_{m})\\
 &  & +\sum_{i<j}(-1)\lambda_{m-1}(e;\alpha_{i}\circ\alpha_{j}+\alpha_{j}\circ\alpha_{i},\cdots\widehat{\alpha_{i}},\cdots\widehat{\alpha_{j}},\cdots)+\frac{1}{m}\sum_{j}\sum_{i\neq j}\omega_{m}(\cdots\alpha_{j}\circ e,\cdots)\\
 & = & (\delta\omega)_{m}(e;\alpha_{1},\cdots\alpha_{m})+(d_{0}\omega)_{m}(e;\alpha_{1},\cdots\alpha_{m})+(\frac{1}{m}+\frac{m-1}{m})\sum_{i}\omega_{m}(\cdots\alpha_{i}\circ e,\cdots)\\
 & = & (d\omega)_{m}(e;\alpha_{1},\cdots\alpha_{m})\\
 & = & 0
\end{eqnarray*}

Step 2:

Suppose $\lambda_{m-1},\cdots,\lambda_{k}(k>0)$ are already defined
so that they satisfy Lambda Conditions, we will construct $\lambda_{k-1}$,
so that it also satisfies Lambda Conditions.

By $\mathfrak{k}$-linearity and the decomposition $\mathcal{E}=\mathcal{H}\oplus\mathcal{X}$,
in order to determine $\lambda_{k-1}$, we only need to define the
value of $\lambda_{k-1}(e_{1},\cdots e_{n+1-2k};\alpha_{1},\cdots\alpha_{k})$
in which each $e_{a}$ is either in $\mathcal{H}$ or in $\mathcal{X}$.

First we let 
\begin{eqnarray}
 &  & \lambda_{k-1}(\beta_{1},\cdots\beta_{l},x_{1},\cdots x_{n+1-2k-l};\alpha_{1},\cdots\alpha_{k-1})\label{eq:1-1}\\
 & \triangleq & \frac{1}{k+l-1}\sum_{1\leq j\leq l}(-1)^{j+1}(\omega_{k}-d_{0}\lambda_{k}-d^{\prime}\lambda_{k})(\beta_{1},\cdots\widehat{\beta_{j}},\cdots\beta_{l},\cdots;\beta_{j},\cdots)\nonumber 
\end{eqnarray}
$\forall\beta_{r},\alpha_{s}\in\mathcal{H},x_{a}\in\mathcal{X}.$

(We call $(\beta,\cdots\beta,x,\cdots x)$ a regular permutation.
)

Note that if $l=0$, we simply let 
\[
\lambda_{k-1}(x_{1},\cdots,x_{n+1-2k};\alpha_{1},\cdots\alpha_{k-1})=0.
\]

For a general permutation $\sigma$, an $x\in\mathcal{X}$ in $\sigma$
is called an irregular element iff there exists at least one element
of $\mathcal{H}$ standing behind $x$ in $\sigma$. The value of
$\lambda_{k-1}(\sigma;\cdots)$ is determined inductively as follows:

If the number of irregular elements in $\sigma$ is $0$, $\sigma$
is a regular permutation. So the value of $\lambda_{k-1}(\sigma;\cdots)$
is determined by Equation \ref{eq:1-1}.

Suppose the value of $\lambda_{k-1}(\sigma;\cdots)$ is already determined
for $\sigma$ with irregular elements less than $t\ (t\geq1)$. Now
for a permutation $\sigma$ with $t$ irregular elements, assume the
last of them is $y\in\mathcal{X}$, and $\sigma=(\bullet,y,\beta_{1},\cdots\beta_{r},x_{1},\cdots x_{s}),\ \beta_{i}\in\mathcal{H},x_{j}\in\mathcal{X}.$
Switching $y$ with $\beta_{1},\cdots\beta_{r}$ one by one, finally
we will get a permutation $\tilde{\sigma}=(\bullet,\beta_{1},\cdots\beta_{r},y,x_{1},\cdots,x_{s})$,
which has $t-1$ irregular elements. The value of $\lambda_{k-1}(\tilde{\sigma};\cdots)$
is already determined. By weak skew-symmetricity we let
\[
\lambda_{k-1}(\sigma;\cdots)\triangleq(-1)^{r}\lambda_{k-1}(\tilde{\sigma};\cdots)+\sum_{1\leq i\leq r}(-1)^{i}\lambda_{k}(\bullet,\beta_{1},\cdots\widehat{\beta_{i}},\cdots\beta_{r},x_{1},\cdots x_{s};\partial\langle y,\beta_{i}\rangle,\cdots).
\]

As a summary, we have extended $\lambda_{k-1}$ from regular permutations
to general permutations by weak skew-symmetricity. The extension could
be written as a formula:
\[
\lambda_{k-1}(\sigma;\cdots)=(\pm1)\lambda_{k-1}(\bar{\sigma};\cdots)+\sum(\pm1)\lambda_{k}(\bullet;\bullet),
\]
where $\bar{\sigma}$ is the regular permutation corresponding to
$\sigma$.We observe that, for different $k$, if we do exactly the
same switchings, then the extension formulas should be similar (each
summand has the same sign, with the subscripts of $\lambda$ modified
correspondingly). For example, if we have an extension formula for
$k$:
\[
\lambda_{k}(\sigma;\cdots)=(\pm1)\lambda_{k}(\bar{\sigma};\cdots)+\sum(\pm1)\lambda_{k+1}(\bullet;\bullet),
\]
then for $k-1$, we have similar formula:
\[
\lambda_{k-1}(\beta,\sigma,x;\cdots)=(\pm1)\lambda_{k-1}(\beta,\bar{\sigma},x;\cdots)+\sum(\pm1)\lambda_{k}(\beta,\bullet,x;\bullet),\quad\forall\beta\in\mathcal{H},x\in\mathcal{X}.
\]

Step 3:

We need to prove that $\lambda_{k-1}$ constructed above satisfies
Lambda Conditions:

Proof of Lambda Condition 1):

First we prove that $\lambda_{k-1}$ for regular permutations is weakly
skew-symmetric up to $\lambda_{k}$ for the arguments in $\mathcal{H}$
and $\mathcal{X}$ respectively. 

When the number of arguments in $\mathcal{H}$ is $0$, the result
is obvious.

Otherwise, for the arguments in $\mathcal{H}$,

\begin{eqnarray*}
 &  & \lambda_{k-1}(\beta_{1},\cdots\beta_{r},\beta_{r+1},\cdots x_{1},\cdots;\alpha_{1},\cdots\alpha_{k-1})+\lambda_{k-1}(\beta_{1},\cdots\beta_{r+1},\beta_{r},\cdots x_{1},\cdots;\cdots)\\
 & = & \frac{1}{k+l-1}((-1)^{r+1}+(-1)^{r})(\omega_{k}-d_{0}\lambda_{k}-d^{\prime}\lambda_{k})(\cdots\widehat{\beta_{r}},\beta_{r+1}\cdots;\beta_{r},\cdots)\\
 &  & +\frac{1}{k+l-1}((-1)^{r}+(-1)^{r+1})(\omega_{k}-d_{0}\lambda_{k}-d^{\prime}\lambda_{k})(\cdots\beta_{r},\widehat{\beta_{r+1}},\cdots;\beta_{r+1},\cdots)\\
 &  & +\frac{1}{k+l-1}\sum_{j\neq r,r+1}(-1)^{j+1}\{(\omega_{k}-d_{0}\lambda_{k}-d^{\prime}\lambda_{k})(\cdots\widehat{\beta_{j}},\cdots\beta_{r},\beta_{r+1}\cdots;\beta_{j},\cdots)\\
 &  & \quad+(\omega_{k}-d_{0}\lambda_{k}-d^{\prime}\lambda_{k})(\cdots\widehat{\beta_{j}},\cdots\beta_{r+1},\beta_{r}\cdots;\beta_{j},\cdots)\}\\
 & = & \frac{1}{k+l-1}\sum_{j\neq r,r+1}(-1)^{j}\\
 &  & \{(d_{0}\lambda_{k})(\cdots\widehat{\beta_{j}},\cdots\beta_{r},\beta_{r+1}\cdots;\beta_{j},\cdots)+(d_{0}\lambda_{k})(\cdots\widehat{\beta_{j}},\cdots\beta_{r+1},\beta_{r}\cdots;\beta_{j},\cdots)\\
 &  & +(d^{\prime}\lambda_{k})(\cdots\widehat{\beta_{j}},\cdots\beta_{r},\beta_{r+1}\cdots;\beta_{j},\cdots)+(d^{\prime}\lambda_{k})(\cdots\widehat{\beta_{j}},\cdots\beta_{r+1},\beta_{r}\cdots;\beta_{j},\cdots)\}\\
 &  & \mbox{(by equation \ref{eq:d0 weak skewsym} and \ref{eq:d prime weak skewsym})}\\
 & = & \frac{1}{k+l-1}\sum_{j\neq r,r+1}(-1)^{j}\{-d_{0}\lambda_{k+1}(\cdots\widehat{\beta_{j}},\cdots\widehat{\beta_{r}},\widehat{\beta_{r+1}},\cdots;\partial\langle\beta_{r},\beta_{r+1}\rangle,\beta_{j},\cdots)\\
 &  & \quad-\lambda_{k}(\partial\langle\beta_{r},\beta_{r+1}\rangle,\cdots\widehat{\beta_{j}},\cdots\widehat{\beta_{r}},\widehat{\beta_{r+1}},\cdots;\beta_{j},\cdots)\\
 &  & \quad-d^{\prime}\lambda_{k+1}(\cdots\widehat{\beta_{j}},\cdots\widehat{\beta_{r}},\widehat{\beta_{r+1}},\cdots;\partial\langle\beta_{r},\beta_{r+1}\rangle,\beta_{j},\cdots)\}\\
 & = & 0
\end{eqnarray*}

For the arguments in $\mathcal{X}$,

\begin{eqnarray*}
 &  & \lambda_{k-1}(\beta_{1}\cdots\beta_{l},x_{1}\cdots x_{a},x_{a+1}\cdots;\cdots)+\lambda_{k-1}(\beta_{1}\cdots\beta_{l},x_{1}\cdots x_{a+1},x_{a}\cdots;\cdots)\\
 &  & +\lambda_{k}(\beta_{1}\cdots\beta_{l},x_{1},\cdots\widehat{x_{a}},\widehat{x_{a+1}},\cdots;\partial\langle x_{a},x_{a+1}\rangle,\cdots)\\
 & = & \frac{1}{k+l-1}\sum_{j}(-1)^{j+1}\{(\omega_{k}-d_{0}\lambda_{k}-d^{\prime}\lambda_{k})(\cdots\widehat{\beta_{j}},\cdots x_{a},x_{a+1}\cdots;\beta_{j},\cdots)\\
 &  & \quad+(\omega_{k}-d_{0}\lambda_{k}-d^{\prime}\lambda_{k})(\cdots\widehat{\beta_{j}},\cdots x_{a+1},x_{a},\cdots;\beta_{j},\cdots)\}\\
 &  & +\frac{1}{k+l}\sum_{j}(-1)^{j+1}(\omega_{k+1}-d_{0}\lambda_{k+1}-d^{\prime}\lambda_{k+1})(\cdots\widehat{\beta_{j}},\cdots\widehat{x_{a}},\widehat{x_{a+1}},\cdots;\beta_{j},\partial\langle x_{a},x_{a+1}\rangle,\cdots)\\
 &  & \mbox{(by equation \ref{eq:d0 weak skewsym} and \ref{eq:d prime weak skewsym})}\\
 & = & \frac{1}{k+l-1}\sum_{j}(-1)^{j+1}\{-\omega_{k+1}(\cdots\widehat{\beta_{j}},\cdots\widehat{x_{a}},\widehat{x_{a+1}},\cdots;\beta_{j},\partial\langle x_{a},x_{a+1}\rangle,\cdots)\\
 &  & \quad+d_{0}\lambda_{k+1}(\cdots\widehat{\beta_{j}},\cdots\widehat{x_{a}},\widehat{x_{a+1}},\cdots;\beta_{j},\partial\langle x_{a},x_{a+1}\rangle,\cdots)\\
 &  & \quad+\lambda_{k}(\langle x_{a},x_{a+1}\rangle,\cdots\widehat{\beta_{j}},\cdots\widehat{x_{a}},\widehat{x_{a+1}},\cdots;\beta_{j},\cdots)\\
 &  & \quad+d^{\prime}\lambda_{k+1}(\cdots\widehat{\beta_{j}},\cdots\widehat{x_{a}},\widehat{x_{a+1}},\cdots;\beta_{j},\partial\langle x_{a},x_{a+1}\rangle,\cdots)\}\\
 &  & +\frac{1}{k+l}\sum_{j}(-1)^{j+1}(\omega_{k+1}-d_{0}\lambda_{k+1}-d^{\prime}\lambda_{k+1})(\cdots\widehat{\beta_{j}},\cdots\widehat{x_{a}},\widehat{x_{a+1}},\cdots;\beta_{j},\partial\langle x_{a},x_{a+1}\rangle,\cdots)\\
 & = & \frac{1}{(k+l-1)(k+l)}\sum_{j}(-1)^{j}(\omega_{k+1}-d_{0}\lambda_{k+1}-d^{\prime}\lambda_{k+1})(\cdots\widehat{\beta_{j}},\cdots\widehat{x_{a}},\widehat{x_{a+1}},\cdots;\beta_{j},\partial\langle x_{a},x_{a+1}\rangle,\cdots)\\
 &  & +\frac{1}{k+l-1}\sum_{j}(-1)^{j+1}\lambda_{k}(\partial\langle x_{a},x_{a+1}\rangle,\cdots\widehat{\beta_{j}},\cdots\widehat{x_{a}},\widehat{x_{a+1}},\cdots;\beta_{j},\cdots)\\
 & = & \frac{1}{(k+l-1)(k+l)}\sum_{j}(-1)^{j+1}\sum_{i<j}(-1)^{i}\\
 &  & \quad(\omega_{k+1}-d_{0}\lambda_{k+1}-d^{\prime}\lambda_{k+1})(\partial\langle x_{a},x_{a+1}\rangle,\cdots\widehat{\beta_{i}},\cdots\widehat{\beta_{j}},\cdots\widehat{x_{a}},\widehat{x_{a+1}},\cdots;\beta_{i},\beta_{j},\cdots)\\
 &  & +\frac{1}{(k+l-1)(k+l)}\sum_{j}(-1)^{j+1}\sum_{i>j}(-1)^{i+1}\\
 &  & \quad(\omega_{k+1}-d_{0}\lambda_{k+1}-d^{\prime}\lambda_{k+1})(\partial\langle x_{a},x_{a+1}\rangle,\cdots\widehat{\beta_{j}},\cdots\widehat{\beta_{i}},\cdots\widehat{x_{a}},\widehat{x_{a+1}},\cdots;\beta_{i},\beta_{j},\cdots)\\
 & = & 0
\end{eqnarray*}

Next, for general permutation $\sigma$, we give the proof in the
following three cases:

(1). $\lambda_{k-1}(\sigma_{1},\beta_{1},\beta_{2},\sigma_{2};\cdots)+\lambda_{k-1}(\sigma_{1},\beta_{2},\beta_{1},\sigma_{2};\cdots)=0,\ \forall\beta_{1},\beta_{2}\in\mathcal{H}$

If every element in $\sigma_{1}$ is in $\mathcal{H}$, then
\begin{eqnarray*}
 &  & \lambda_{k-1}(\sigma_{1},\beta_{1},\beta_{2},\sigma_{2};\cdots)+\lambda_{k-1}(\sigma_{1},\beta_{2},\beta_{1},\sigma_{2};\cdots)\\
 & = & (\pm1)\lambda_{k-1}(\sigma_{1},\beta_{1},\beta_{2},\bar{\sigma_{2}};\cdots)+\sum(\pm1)\lambda_{k}(\sigma_{1},\beta_{1},\beta_{2},\bullet;\bullet)\\
 &  & +(\pm1)\lambda_{k-1}(\sigma_{1},\beta_{2},\beta_{1},\bar{\sigma_{2}};\cdots)+\sum(\pm1)\lambda_{k}(\sigma_{1},\beta_{2},\beta_{1},\bullet;\bullet)\\
 & = & (\pm1)\big(\lambda_{k-1}(\sigma_{1},g_{1},g_{2},\bar{\sigma_{2}};\cdots)+\lambda_{k-1}(\sigma_{1},g_{2},g_{1},\bar{\sigma_{2}};\cdots)\big)\\
 &  & +\sum(\pm1)\big(\lambda_{k}(\sigma_{1},g_{1},g_{2},\bullet;\bullet)+\lambda_{k}(\sigma_{1},g_{2},g_{1},\bullet;\bullet)\big)\\
 & = & 0
\end{eqnarray*}

Now suppose (1) holds for $\sigma_{1}$ containing at most $m$ elements
in $\mathcal{X}$, consider the case when $\sigma_{1}$ contains $m+1$
elements in $\mathcal{X}$, suppose $x$ is the last one of them,
move $x$ to the last position in $\sigma_{1}$ and denote the elements
in front of $x$ as $\tilde{\sigma_{1}}$, $\tilde{\sigma_{1}}$ contains
$m$ elements in $X$.
\begin{eqnarray*}
 &  & \lambda_{k-1}(\sigma_{1},\beta_{1},\beta_{2},\sigma_{2};\cdots)+\lambda_{k-1}(\sigma_{1},\beta_{2},\beta_{1},\sigma_{2};\cdots)\\
 & = & (\pm1)\lambda_{k-1}(\sigma_{1},\beta_{1},\beta_{2},\bar{\sigma_{2}};\cdots)+\sum(\pm1)\lambda_{k}(\sigma_{1},\beta_{1},\beta_{2},\bullet;\bullet)\\
 &  & +(\pm1)\lambda_{k-1}(\sigma_{1},\beta_{2},\beta_{1},\bar{\sigma_{2}};\cdots)+\sum(\pm1)\lambda_{k}(\sigma_{1},\beta_{2},\beta_{1},\bullet;\bullet)\\
 & = & (\pm1)\big(\lambda_{k-1}(\sigma_{1},\beta_{1},\beta_{2},\bar{\sigma_{2}};\cdots)+\lambda_{k-1}(\sigma_{1},\beta_{2},\beta_{1},\bar{\sigma_{2}};\cdots)\big)\\
 & = & (\pm1)\big((\pm1)\lambda_{k-1}(\tilde{\sigma_{1}},x,\beta_{1},\beta_{2},\bar{\sigma_{2}};\cdots)+\sum(\pm1)\lambda_{k}(\bullet,\beta_{1},\beta_{2},\bar{\sigma_{2}};\bullet)\\
 &  & \quad+(\pm1)\lambda_{k-1}(\tilde{\sigma_{1}},x,\beta_{2},\beta_{1},\bar{\sigma_{2}};\cdots)+\sum(\pm1)\lambda_{k}(\bullet,\beta_{2},\beta_{1},\bar{\sigma_{2}};\bullet)\big)\\
 & = & (\pm1)\big(\lambda_{k-1}(\tilde{\sigma_{1}},x,\beta_{1},\beta_{2},\bar{\sigma_{2}};\cdots)+\lambda_{k-1}(\tilde{\sigma_{1}},x,\beta_{2},\beta_{1},\bar{\sigma_{2}};\cdots)\big)\\
 & = & (\pm1)\big(-\lambda_{k}(\tilde{\sigma_{1}},\beta_{2},\bar{\sigma_{2}};\partial\langle x,\beta_{1}\rangle,\cdots)+\lambda_{k}(\tilde{\sigma_{1}},\beta_{1},\bar{\sigma_{2}};\partial\langle x,\beta_{2}\rangle,\cdots)+\lambda_{k-1}(\tilde{\sigma_{1}},\beta_{1},\beta_{2},x,\bar{\sigma_{2}};\cdots)\\
 &  & \ -\lambda_{k}(\tilde{\sigma_{1}},\beta_{1},\bar{\sigma_{2}};\partial\langle x,\beta_{2}\rangle,\cdots)+\lambda_{k}(\tilde{\sigma_{1}},\beta_{2},\bar{\sigma_{2}};\partial\langle x,\beta_{1}\rangle,\cdots)+\lambda_{k-1}(\tilde{\sigma_{1}},\beta_{2},\beta_{1},x,\bar{\sigma_{2}};\cdots)\big)\\
 & = & (\pm1)\big(\lambda_{k-1}(\tilde{\sigma_{1}},\beta_{1},\beta_{2},x,\bar{\sigma_{2}};\cdots)+\lambda_{k-1}(\tilde{\sigma_{1}},\beta_{2},\beta_{1},x,\bar{\sigma_{2}};\cdots)\big)\\
 & = & 0
\end{eqnarray*}

By mathematical induction, (1) is proved.

(2). $\lambda_{k-1}(\sigma_{1},\beta,y,\sigma_{2};\cdots)+\lambda_{k-1}(\sigma_{1},y,\beta,\sigma_{2};\cdots)=-\lambda_{k}(\sigma_{1},\sigma_{2};\partial\langle\beta,y\rangle,\cdots),\ \forall\beta\in\mathcal{H},\ y\in\mathcal{X}$

\begin{eqnarray*}
 &  & \lambda_{k-1}(\sigma_{1},\beta,y,\sigma_{2};\cdots)+\lambda_{k-1}(\sigma_{1},y,\beta,\sigma_{2};\cdots)\\
 & = & (\pm1)\lambda_{k-1}(\sigma_{1},\beta,y,\bar{\sigma_{2}};\cdots)+\sum(\pm1)\lambda_{k}(\sigma_{1},\beta,y,\bullet;\bullet)\\
 &  & +(\pm1)\lambda_{k-1}(\sigma_{1},y,\beta,\bar{\sigma_{2}};\cdots)+\sum(\pm1)\lambda_{k}(\sigma_{1},\beta,g,\bullet;\bullet)\\
 & = & (\pm1)\big(\lambda_{k-1}(\sigma_{1},\beta,y,\bar{\sigma_{2}};\cdots)+\lambda_{k-1}(\sigma_{1},y,\beta,\bar{\sigma_{2}};\cdots)\big)\\
 &  & +\sum(\pm1)\big(\lambda_{k}(\sigma_{1},\beta,y,\bullet;\bullet)+\lambda_{k}(\sigma_{1},y,\beta,\bullet;\bullet)\big)\\
 & = & (\pm1)\big(\lambda_{k-1}(\sigma_{1},\beta,y,\bar{\sigma_{2}};\cdots)+(-\lambda_{k}(\sigma_{1},\bar{\sigma_{2}};\partial\langle\beta,y\rangle,\cdots)-\lambda_{k-1}(\sigma_{1},\beta,y,\bar{\sigma_{2}};\cdots))\big)\\
 &  & -\sum(\pm1)\lambda_{k+1}(\sigma_{1},\bullet;\partial\langle\beta,y\rangle,\bullet)\\
 & = & -\big((\pm1)\lambda_{k}(\sigma_{1},\bar{\sigma_{2}};\partial\langle\beta,y\rangle,\cdots)+\sum(\pm1)\lambda_{k+1}(\sigma_{1},\bullet;\partial\langle\beta,y\rangle,\bullet)\big)\\
 & = & -\lambda_{k}(\sigma_{1},\sigma_{2};\partial\langle\beta,y\rangle,\cdots)
\end{eqnarray*}

(3). $\lambda_{k-1}(\sigma_{1},y_{1},y_{2},\sigma_{2};\cdots)+\lambda_{k-1}(\sigma_{1},y_{2},y_{1},\sigma_{2};\cdots)=-\lambda_{k}(\sigma_{1},\sigma_{2};\partial\langle y_{1},y_{2}\rangle,\cdots)$,
$\forall y_{1},y_{2}\in X$.

Suppose $\bar{\sigma_{2}}=(\beta_{1},\cdots\beta_{a},x_{1},\cdots x_{b})$,
then
\begin{eqnarray*}
 &  & \lambda_{k-1}(\sigma_{1},y_{1},y_{2},\bar{\sigma_{2}};\cdots)+\lambda_{k-1}(\sigma_{1},y_{2},y_{1},\bar{\sigma_{2}};\cdots)\\
 & = & \big(\sum_{1\leq i\leq a}(-1)^{i}\lambda_{k}(\sigma_{1},y_{1},\beta_{1},\cdots\widehat{\beta_{i}},\cdots\beta_{a},x_{1},\cdots x_{b};\partial\langle y_{2},\beta_{i}\rangle,\cdots)\\
 &  & \quad+(-1)^{a}\lambda_{k-1}(\sigma_{1},y_{1},\beta_{1},\cdots\beta_{a},y_{2},x_{1},\cdots x_{b};\cdots)\big)\\
 &  & +\big(\sum_{1\leq j\leq a}(-1)^{j}\lambda_{k}(\sigma_{1},y_{2},\beta_{1},\cdots\widehat{\beta_{j}},\cdots\beta_{a},x_{1},\cdots x_{b};\partial\langle y_{1},\beta_{j}\rangle,\cdots)\\
 &  & \quad+(-1)^{a}\lambda_{k-1}(\sigma_{1},y_{2},\beta_{1},\cdots\beta_{a},y_{1},x_{1},\cdots x_{b};\cdots)\big)\\
 & = & \sum_{i}(-1)^{i}\big(\sum_{1\leq j<i}(-1)^{j}\lambda_{k+1}(\sigma_{1},\beta_{1},\cdots\widehat{\beta_{j}},\cdots\widehat{\beta_{i}},\cdots\beta_{a},x_{1},\cdots x_{b};\partial\langle y_{1},\beta_{j}\rangle,\partial\langle y_{2},\beta_{i}\rangle,\cdots)\\
 &  & \quad+\sum_{j>i}(-1)^{j+1}\lambda_{k+1}(\sigma_{1},\beta_{1},\cdots\widehat{\beta_{i}},\cdots\widehat{\beta_{j}},\cdots\beta_{a},x_{1}\cdots x_{b};\partial\langle y_{1},\beta_{j}\rangle,\partial\langle y_{2},\beta_{i}\rangle,\cdots)\\
 &  & \quad+(-1)^{a+1}\lambda_{k}(\sigma_{1},\beta_{1},\cdots\widehat{\beta_{i}},\cdots\beta_{a},y_{1},x_{1},\cdots x_{b};\partial\langle y_{2},\beta_{i}\rangle,\cdots)\big)\\
 &  & +(-1)^{a}\big(\sum_{1\leq j\leq a}(-1)^{j}\lambda_{k}(\sigma_{1},\beta_{1},\cdots\widehat{\beta_{j}},\cdots\beta_{a},y_{2},x_{1},\cdots x_{b};\partial\langle y_{1},\beta_{j}\rangle,\cdots)\\
 &  & \quad+(-1)^{a}\lambda_{k-1}(\sigma_{1},\beta_{1},\cdots\beta_{a},y_{1},y_{2},x_{1},\cdots x_{b};\cdots)\big)\\
 &  & +\sum_{j}(-1)^{j}\big(\sum_{1\leq i<j}(-1)^{i}\lambda_{k+1}(\sigma_{1},\beta_{1},\cdots\widehat{\beta_{i}},\cdots\widehat{\beta_{j}},\cdots\beta_{a},x_{1},\cdots x_{b};\partial\langle y_{1},\beta_{j}\rangle,\partial\langle y_{2},\beta_{i}\rangle,\cdots)\\
 &  & \quad+\sum_{i>j}(-1)^{i+1}\lambda_{k+1}(\sigma_{1},\beta_{1},\cdots\widehat{\beta_{j}},\cdots\widehat{\beta_{i}},\cdots\beta_{a},x_{1},\cdots x_{b};\partial\langle y_{1},\beta_{j}\rangle,\partial\langle y_{2},\beta_{i}\rangle,\cdots)\\
 &  & \quad+(-1)^{a+1}\lambda_{k}(\sigma_{1},\beta_{1},\cdots\widehat{\beta_{j}},\cdots\beta_{a},y_{2},x_{1},\cdots x_{b};\partial\langle y_{1},\beta_{j}\rangle,\cdots)\big)\\
 &  & +(-1)^{a}\big(\sum_{1\leq i\leq a}(-1)^{i}\lambda_{k}(\sigma_{1},\beta_{1},\cdots\widehat{\beta_{i}},\cdots\beta_{a},y_{1},x_{1},\cdots x_{b};\partial\langle y_{2},\beta_{i}\rangle,\cdots)\\
 &  & \quad+(-1)^{a}\lambda_{k-1}(\sigma_{1},\beta_{1},\cdots\beta_{a},y_{2},y_{1},x_{1},\cdots x_{b};\cdots)\big)\\
 & = & \lambda_{k-1}(\sigma_{1},\beta_{1},\cdots\beta_{a},y_{1},y_{2},x_{1},\cdots x_{b};\cdots)+\lambda_{k-1}(\sigma_{1},\beta_{1},\cdots\beta_{a},y_{2},y_{1},x_{1},\cdots x_{b};\cdots)
\end{eqnarray*}

So
\begin{eqnarray*}
 &  & \lambda_{k-1}(\sigma_{1},y_{1},y_{2},\sigma_{2};\cdots)+\lambda_{k-1}(\sigma_{1},y_{2},y_{1},\sigma_{2};\cdots)\\
 & = & (\pm1)\lambda_{k-1}(\sigma_{1},y_{1},y_{2},\bar{\sigma_{2}};\cdots)+\sum(\pm1)\lambda_{k}(\sigma_{1},y_{1},y_{2},\bullet;\bullet)\\
 &  & +(\pm1)\lambda_{k-1}(\sigma_{1},y_{2},y_{1},\bar{\sigma_{2}};\cdots)+\sum(\pm1)\lambda_{k}(\sigma_{1},y_{2},y_{1},\bullet;\bullet)\\
 & = & (\pm1)\big(\lambda_{k-1}(\sigma_{1},y_{1},y_{2},\bar{\sigma_{2}};\cdots)+\lambda_{k-1}(\sigma_{1},y_{2},y_{1},\bar{\sigma_{2}};\cdots)\big)\\
 &  & +\sum(\pm1)\big(\lambda_{k}(\sigma_{1},y_{1},y_{2},\bullet;\bullet)+\lambda_{k}(\sigma_{1},y_{2},y_{1},\bullet;\bullet)\big)\\
 & = & (\pm1)\big(\lambda_{k-1}(\sigma_{1},\beta_{1},\cdots\beta_{a},y_{1},y_{2},x_{1},\cdots x_{b};\cdots)+\lambda_{k-1}(\sigma_{1},\beta_{1},\cdots\beta_{a},y_{2},y_{1},x_{1},\cdots x_{b};\cdots)\big)\\
 &  & -\sum(\pm1)\lambda_{k+1}(\sigma_{1},\widehat{y_{1}},\widehat{y_{2}},\bullet;\partial\langle y_{1},y_{2}\rangle,\bullet)\\
 &  & \mbox{(denote by \ensuremath{\mu}the permutation \ensuremath{(\sigma_{1},\beta_{1},\cdots,\beta_{a})})}\\
 & = & (\pm1)\big((\pm1)\lambda_{k-1}(\bar{\mu},y_{1},y_{2},x_{1},\cdots x_{b};\cdots)+\sum(\pm1)\lambda_{k}(\bullet,y_{1},y_{2},x_{1},\cdots x_{b};\bullet)\\
 &  & \quad+(\pm1)\lambda_{k-1}(\bar{\mu},y_{2},y_{1},x_{1},\cdots,x_{b};\cdots)+\sum(\pm1)\lambda_{k}(\bullet,y_{2},y_{1},x_{1},\cdots,x_{b};\bullet)\big)\\
 &  & -\sum(\pm1)\lambda_{k+1}(\sigma_{1},\widehat{y_{1}},\widehat{y_{2}},\bullet;\partial\langle y_{1},y_{2}\rangle,\bullet)\\
 & = & -(\pm1)\big((\pm1)\lambda_{k}(\bar{\mu},\widehat{y_{1}},\widehat{y_{2}},x_{1},\cdots;\partial\langle y_{1},y_{2}\rangle,\cdots)+\sum(\pm1)\lambda_{k+1}(\bullet,\widehat{y_{1}},\widehat{y_{2}},x_{1},\cdots;\partial\langle y_{1},y_{2}\rangle,\bullet)\big)\\
 &  & -\sum(\pm1)\lambda_{k+1}(\sigma_{1},\widehat{y_{1}},\widehat{y_{2}},\bullet;\partial\langle y_{1},y_{2}\rangle,\bullet)\\
 & = & -\big((\pm1)\lambda_{k}(\mu,\widehat{y_{1}},\widehat{y_{2}},x_{1},\cdots x_{b};\partial\langle y_{1},y_{2}\rangle,\cdots)+\sum(\pm1)\lambda_{k+1}(\sigma_{1},\widehat{y_{1}},\widehat{y_{2}},\bullet;\partial\langle y_{1},y_{2}\rangle,\bullet)\big)\\
 & = & -\big((\pm1)\lambda_{k}(\sigma_{1},\widehat{y_{1}},\widehat{y_{2}},\bar{\sigma_{2}};\partial\langle y_{1},y_{2}\rangle,\cdots)+\sum(\pm1)\lambda_{k+1}(\sigma_{1},\widehat{y_{1}},\widehat{y_{2}},\bullet;\partial\langle y_{1},y_{2}\rangle,\bullet)\big)\\
 & = & -\lambda_{k}(\sigma_{1},\widehat{y_{1}},\widehat{y_{2}},\sigma_{2};\partial\langle y_{1},y_{2}\rangle,\cdots)
\end{eqnarray*}

Combining (1)(2)(3) above, Lambda condition 1) for $\lambda_{k-1}$
is proven.

Proof of Lambda Condition 2):

Since the weak skew-symmetricity of $\lambda_{k-1}$ is already proven,
we only need to prove that $\lambda_{k-1}$ is $R$-linear in the
last argument of $\mathcal{E}$.

When the last argument is in $\mathcal{X}$:
\begin{eqnarray*}
 &  & \lambda_{k-1}(\sigma,fx;\cdots)-f\lambda_{k-1}(\sigma,x;\cdots)\\
 & = & \big((\pm1)\lambda_{k-1}(\bar{\sigma},fx;\cdots)+\sum(\pm1)\lambda_{k}(\bullet,fx;\bullet)\big)-f\big((\pm1)\lambda_{k-1}(\bar{\sigma},x;\cdots)+\sum(\pm1)\lambda_{k}(\bullet,x;\bullet)\big)\\
 &  & \mbox{(suppose \ensuremath{\bar{\sigma}=(\beta_{1},\cdots\beta_{l},x_{1},\cdots x_{n-2k-l})})}\\
 & = & \frac{\pm1}{k+l-1}\sum_{j}(-1)^{j+1}\{(\omega_{k}-d_{0}\lambda_{k}-d^{\prime}\lambda_{k})(\cdots\widehat{\beta_{j}},\cdots fx;\beta_{j},\cdots)\\
 &  & \quad-f(\omega_{k}-d_{0}\lambda_{k}-d^{\prime}\lambda_{k})(\cdots\widehat{\beta_{j}},\cdots x;\beta_{j},\cdots)\}\\
 & = & 0\mbox{ (By Equation \ref{eq:d0 weak R linear} and \ref{eq:d prime weak R linear})}.
\end{eqnarray*}

When the last argument is in $\mathcal{H}$:
\begin{eqnarray*}
 &  & \lambda_{k-1}(\sigma,f\beta;\cdots)-f\lambda_{k-1}(\sigma,\beta;\cdots)\\
 & = & \big((\pm1)\lambda_{k-1}(\bar{\sigma},f\beta;\cdots)+\sum(\pm1)\lambda_{k}(\bullet,f\beta;\bullet)\big)-f\big((\pm1)\lambda_{k-1}(\bar{\sigma},\beta;\cdots)+\sum(\pm1)\lambda_{k}(\bullet,\beta;\bullet)\big)\\
 &  & \mbox{(suppose \ensuremath{\bar{\sigma}=(\beta_{1},\cdots\beta_{l-1},x_{1},\cdots x_{n+1-2k-l})})}\\
 & = & (\pm1)\big((-1)^{n+1+l}\lambda_{k-1}(\beta_{1},\cdots\beta_{l-1},f\beta,x_{1},\cdots;\cdots)\\
 &  & \quad+\sum_{a}(-1)^{a+n+l}\lambda_{k}(\beta_{1},\cdots\beta_{l-1},x_{1},\cdots\widehat{x_{a}},\cdots x_{n+1-2k-l};\partial\langle x_{a},f\beta\rangle,\cdots)\big)\\
 &  & -(\pm1)f\big((-1)^{n+1+l}\lambda_{k-1}(\beta_{1},\cdots\beta_{l-1},\beta,x_{1},\cdots;\cdots)\\
 &  & \quad+\sum_{a}(-1)^{a+n+l}\lambda_{k}(\beta_{1},\cdots\beta_{l-1},x_{1},\cdots\widehat{x_{a}},\cdots x_{n+1-2k-l};\partial\langle x_{a},\beta\rangle,\cdots)\big)\\
 & = & \frac{(\pm1)(-1)^{n+1+l}}{k+l-1}\sum_{j}(-1)^{j+1}\big((\omega_{k}-d_{0}\lambda_{k}-d^{\prime}\lambda_{k})(\cdots\widehat{\beta_{j}},\cdots f\beta,\cdots;\beta_{j},\cdots)\\
 &  & \quad-f(\omega_{k}-d_{0}\lambda_{k}-d^{\prime}\lambda_{k})(\cdots\widehat{\beta_{j}},\cdots\beta,\cdots;\beta_{j},\cdots)\big)\\
 &  & +\frac{(\pm1)(-1)^{n}}{k+l-1}\big((\omega_{k}-d_{0}\lambda_{k}-d^{\prime}\lambda_{k})(\cdots\widehat{\beta},\cdots;f\beta,\cdots)-f(\omega_{k}-d_{0}\lambda_{k}-d^{\prime}\lambda_{k})(\cdots\widehat{\beta},\cdots;\beta,\cdots)\big)\\
 &  & +(\pm1)\sum_{a}(-1)^{a+n+l}\langle x_{a},\beta\rangle\lambda_{k}(\cdots\widehat{\beta},\cdots\widehat{x_{a}},\cdots;\partial f,\cdots)\\
 &  & \mbox{(By Equation \ref{eq:d0 weak R linear}, \ref{eq:d prime weak R linear} and \ref{eq:d R linear})}\\
 & = & \frac{\pm1}{k+l-1}\sum_{j,a}(-1)^{n+l+j+a}\langle\beta,x_{a}\rangle\big((\omega_{k+1}-d_{0}\lambda_{k+1}-d^{\prime}\lambda_{k+1})(\cdots\widehat{\beta_{j}},\cdots\widehat{\beta},\cdots\widehat{x_{a}},\cdots;\partial f,\beta_{j},\cdots)\\
 &  & \quad-\lambda_{k}(\partial f,\cdots\widehat{\beta_{j}},\cdots\widehat{\beta},\cdots\widehat{x_{a}},\cdots;\beta_{j},\cdots)\big)\\
 &  & +\frac{\pm1}{k+l-1}\sum_{a}(-1)^{n+1+l+a}\langle\beta,x_{a}\rangle\lambda_{k}(\cdots\widehat{\beta},\cdots\widehat{x_{a}},\cdots;\partial f,\cdots)\\
 &  & +\frac{\pm1}{k+l-1}\sum_{j,a}(-1)^{a+n+l+j+1}\langle\beta,x_{a}\rangle(\omega_{k+1}-d_{0}\lambda_{k+1}-d^{\prime}\lambda_{k+1})(\cdots\widehat{\beta_{j}},\cdots\widehat{\beta},\cdots\widehat{x_{a}},\cdots;\partial f,\beta_{j},\cdots)\\
 & = & \frac{\pm1}{k+l-1}\sum_{a}\sum_{0\leq j\leq l-1}(-1)^{n+l+j+a+1}\langle\beta,x_{a}\rangle\lambda_{k}(\beta_{0}(\triangleq\partial f),\cdots\widehat{\beta_{j}},\cdots\beta_{l-1},\cdots\widehat{x_{a}},\cdots;\beta_{j},\cdots)\\
 & = & \frac{\pm1}{(k+l-1)^{2}}\sum_{a}(-1)^{n+l+a+1}\langle\beta,x_{a}\rangle\\
 &  & \big(\sum_{0\leq i<j\leq l-1}(-1)^{i+j}(\omega_{k+1}-d_{0}\lambda_{k+1}-d^{\prime}\lambda_{k+1})(\beta_{0},\cdots\widehat{\beta_{i}},\cdots\widehat{\beta_{j}},\cdots\widehat{x_{a}},\cdots;\beta_{i},\beta_{j},\cdots)\\
 &  & +\sum_{0\leq j<i\leq l-1}(-1)^{i+j+1}(\omega_{k+1}-d_{0}\lambda_{k+1}-d^{\prime}\lambda_{k+1})(\beta_{0},\cdots\widehat{\beta_{j}},\cdots\widehat{\beta_{i}},\cdots\widehat{x_{a}},\cdots;\beta_{i},\beta_{j},\cdots)\big)\\
 & = & 0
\end{eqnarray*}

Proof of Lambda Condition 3):

\begin{eqnarray*}
 &  & \lambda_{k-1}(\beta_{1},\cdots\beta_{l},x_{1},\cdots x_{n+1-2k-l};f\alpha_{1},\cdots\alpha_{k-1})-f\lambda_{k-1}(\cdots;\alpha_{1},\cdots\alpha_{k-1})\\
 & = & \frac{1}{k+l-1}\sum_{j}(-1)^{j+1}\big((\omega_{k}-d_{0}\lambda_{k}-d^{\prime}\lambda_{k})(\cdots\widehat{\beta_{j}},\cdots;\beta_{j},f\alpha_{1},\cdots)\\
 &  & \quad-f(\omega_{k}-d_{0}\lambda_{k}-d^{\prime}\lambda_{k})(\cdots\widehat{\beta_{j}},\cdots;\beta_{j},\alpha_{1},\cdots)\big)\\
 &  & \mbox{(By Equation \ref{eq:d R linear})}\\
 & = & \frac{1}{k+l-1}\sum_{j}(-1)^{j}\sum_{a}(-1)^{l+a+1}\langle x_{a},\alpha_{1}\rangle\lambda_{k}(\cdots\widehat{\beta_{j}},\cdots\widehat{x_{a}},\cdots;\beta_{j},\partial f,\widehat{\alpha_{1}},\cdots)\\
 & = & \frac{1}{(k+l-1)^{2}}\sum_{a}(-1)^{l+a+1}\langle x_{a},\alpha_{1}\rangle\\
 &  & \big(\sum_{i<j}(-1)^{j+i+1}(\omega_{k+1}-d_{0}\lambda_{k+1}-d^{\prime}\lambda_{k+1})(\cdots\widehat{\beta_{i}},\cdots\widehat{\beta_{j}},\cdots\widehat{x_{a}},\cdots;\beta_{i},\beta_{j},\partial f,\widehat{\alpha_{1}},\cdots)\\
 &  & +\sum_{i>j}(-1)^{j+i}(\omega_{k+1}-d_{0}\lambda_{k+1}-d^{\prime}\lambda_{k+1})(\cdots\widehat{\beta_{j}},\cdots\widehat{\beta_{i}},\cdots\widehat{x_{a}},\cdots;\beta_{i},\beta_{j},\partial f,\widehat{\alpha_{1}},\cdots)\big)\\
 & = & 0
\end{eqnarray*}

Proof of Lambda Condition 4):

For regular permutations: 
\begin{eqnarray*}
 &  & (\delta\lambda)_{k}(\beta_{1},\cdots\beta_{l},x_{1},\cdots x_{n-2k-l};\alpha_{1},\cdots\alpha_{k})\\
 & = & \sum_{i}\lambda_{k-1}(\alpha_{i},\beta_{1},\cdots\beta_{l},x_{1},\cdots x_{n-2k-l};\cdots\widehat{\alpha_{i}},\cdots)\\
 & = & \frac{1}{k+l}\sum_{i}(\omega_{k}-d_{0}\lambda_{k}-d^{\prime}\lambda_{k})(\beta_{1},\cdots\beta_{l},\cdots;\alpha_{1},\cdots\alpha_{k})\\
 &  & +\frac{1}{k+l}\sum_{i,j}(-1)^{j}(\omega_{k}-d_{0}\lambda_{k}-d^{\prime}\lambda_{k})(\alpha_{i},\beta_{1}\cdots\widehat{\beta_{j}},\cdots\beta_{l}\cdots;\beta_{j},\cdots\widehat{\alpha_{i}},\cdots)\\
 &  & \mbox{(By Lambda Condition 5) for \ensuremath{\lambda_{k}})}\\
 & = & \frac{k}{k+l}(\omega_{k}-d_{0}\lambda_{k}-d^{\prime}\lambda_{k})(\beta_{1},\cdots\beta_{l},\cdots;\alpha_{1},\cdots\alpha_{k})\\
 &  & +\frac{1}{k+l}\sum_{j}(-1)^{j}(-1)(\omega_{k}-d_{0}\lambda_{k}-d^{\prime}\lambda_{k})(\beta_{j},\beta_{1}\cdots\widehat{\beta_{j}},\cdots\beta_{l}\cdots;\alpha_{1},\cdots\alpha_{k})\\
 & = & \frac{k+l}{k+l}(\omega_{k}-d_{0}\lambda_{k}-d^{\prime}\lambda_{k})(\beta_{1},\cdots\beta_{l},\cdots;\alpha_{1},\cdots\alpha_{k})
\end{eqnarray*}

So $\omega_{k}=(d\lambda)_{k}$ for regular permutation $(\beta_{1},\cdots\beta_{l},x_{1},\cdots x_{n-2k-l})$.

Since $\omega_{k}$ and $(d\lambda)_{k}$ are both weakly skew-symmetric
up to $\omega_{k+1}=(d\lambda)_{k+1}$, so $\omega_{k}=(d\lambda)_{k}$
holds for general permutations.

Proof of Lambda Condition 5):

We only need to prove the following:
\begin{eqnarray*}
 &  & \sum_{i}(\omega_{k-1}-d_{0}\lambda_{k-1}-d^{\prime}\lambda_{k-1})(\alpha_{i},e_{1},\cdots,e_{n+1-2k};\alpha_{1},\cdots\widehat{\alpha_{i}},\cdots\alpha_{k})\\
 & = & (d\omega)_{k}(e_{1},\cdots,e_{n+1-2k};\alpha_{1},\cdots\alpha_{k})
\end{eqnarray*}

or equivalently,

\begin{eqnarray*}
 &  & (d_{0}\omega_{k}+d^{\prime}\omega_{k})(e_{1},\cdots e_{n+1-2k};\alpha_{1},\cdots\alpha_{k})+\sum_{i}(d_{0}\lambda_{k-1}+d^{\prime}\lambda_{k-1})(\alpha_{i},e_{1},\cdots;\cdots\widehat{\alpha_{i}},\cdots)\\
 & = & (d_{0}\omega_{k}+d^{\prime}\omega_{k})(e_{1},\cdots e_{n+1-2k};\alpha_{1},\cdots\alpha_{k})\\
 &  & +\sum_{i,a}(-1)^{a}\nabla_{e_{a}}\lambda_{k-1}(\alpha_{i},\cdots\widehat{e_{a}},\cdots;\cdots\widehat{\alpha_{i}},\cdots)+\sum_{i,a}(-1)\lambda_{k-1}(\cdots\widehat{e_{a}},\alpha_{i}\circ e_{a},\cdots;\cdots\widehat{\alpha_{i}},\cdots)\\
 &  & +\sum_{i,a<b}(-1)^{a+1}\lambda_{k-1}(\alpha_{i},\cdots\widehat{e_{a}},\cdots e_{a}\circ e_{b},\cdots;\cdots\widehat{\alpha_{i}},\cdots)+\sum_{j\neq i}\lambda_{k-1}(\cdots;\cdots\widehat{\alpha_{i}},\cdots\alpha_{j}\circ\alpha_{i},\cdots)\\
 &  & +\sum_{j\neq i}\lambda_{k-1}(\cdots;\cdots\widehat{\alpha_{i}},\cdots\alpha_{j}\circ\alpha_{i},\cdots)+\sum_{a,j\neq i}(-1)^{a}\lambda_{k-1}(\alpha_{i},\cdots\widehat{e_{a}},\cdots;\cdots\widehat{\alpha_{i}},\cdots\alpha_{j}\circ e_{a},\cdots)\\
 & = & (d_{0}\omega_{k}+d^{\prime}\omega_{k})(e_{1},\cdots e_{n+1-2k};\alpha_{1},\cdots\alpha_{k})+\sum_{a}(-1)^{a}\nabla_{e_{a}}(\omega_{k}-d_{0}\lambda_{k}-d^{\prime}\lambda_{k})(\cdots\widehat{e_{a}},\cdots;\cdots)\\
 &  & +\sum_{a<b}(-1)^{a+1}(\omega_{k}-d_{0}\lambda_{k}-d^{\prime}\lambda_{k})(\cdots\widehat{e_{a}},\cdots e_{a}\circ e_{b},\cdots)-\sum_{a,i}\lambda_{k-1}(\cdots\widehat{e_{a}},\alpha_{i}\circ e_{a},\cdots;\cdots\widehat{\alpha_{i}},\cdots)\\
 &  & +\sum_{i<j}\lambda_{k-1}(e_{1},\cdots e_{n+1-2k};\alpha_{i}\circ\alpha_{j}+\alpha_{j}\circ\alpha_{i},\cdots\widehat{\alpha_{i}},\cdots\widehat{\alpha_{j}},\cdots)\\
 &  & +\big(\sum_{a,j}(-1)^{a}(\omega_{k}-d_{0}\lambda_{k}-d^{\prime}\lambda_{k})(\cdots\widehat{e_{a}},\cdots;\cdots\alpha_{j}\circ e_{a},\cdots)\\
 &  & \quad-\sum_{a,j}(-1)^{a}\lambda_{k-1}(\alpha_{j}\circ e_{a},\cdots\widehat{e_{a}},\cdots;\cdots\widehat{\alpha_{j}},\cdots)\big)\\
 & = & ((d_{0}^{2}+d_{0}\circ d^{\prime})\lambda)_{k}(e_{1},\cdots;\alpha_{1},\cdots)-\sum_{a,i}\lambda_{k-1}(\cdots\widehat{e_{a}},\alpha_{i}\circ e_{a},\cdots;\cdots\widehat{\alpha_{i}},\cdots)\\
 &  & +((d^{\prime}\circ d_{0}+d^{\prime2})\lambda)_{k}(e_{1},\cdots;\alpha_{1},\cdots)+\sum_{a,j}(-1)^{a+1}\lambda_{k-1}(\alpha_{j}\circ e_{a},\cdots\widehat{e_{a}},\cdots;\cdots\widehat{\alpha_{j}},\cdots)\\
 & = & ((d_{0}+d^{\prime}+\delta)^{2}\lambda)_{k}(e_{1},\cdots e_{n+1-2k};\alpha_{1},\cdots\alpha_{k})\\
 & = & 0
\end{eqnarray*}

Thus $\lambda_{k-1}$ satisfies all Lambda Conditions.

Step 4:

By mathematical induction, finally we obtain $(\lambda_{0},\cdots,\lambda_{m-1})$
with each $\lambda_{p}$ satisfying Lambda Conditions. 

Let $\lambda\triangleq(\lambda_{0},\cdots,\lambda_{m-1})$, Lambda
Conditions 1), 2), 3) imply that $\lambda$ is a cochain in $C^{n-1}(\mathcal{E},\mathcal{H},\mathcal{V})$.

Let $\eta\triangleq\omega-d\lambda\in C^{n}(\mathcal{E},\mathcal{H},\mathcal{V})$. 

By Lambda Condition 4), $\eta=(\omega_{0}-(d\lambda)_{0},0,\cdots,0)=(\omega_{0}-d_{0}\lambda_{0}-d^{\prime}\lambda_{0},0,\cdots,0)$. 

By Lambda Condition 5), $(\omega_{0}-d_{0}\lambda_{0}-d^{\prime}\lambda_{0})(\alpha,e_{1},\cdots,e_{n-1})=0,\ \forall\alpha\in\mathcal{H},\ e_{i}\in\mathcal{E}.$

So $\iota_{\alpha}\eta_{0}=\iota_{\alpha}(\omega_{0}-d_{0}\lambda_{0}-d^{\prime}\lambda_{0})=0,\ \forall\alpha\in\mathcal{H}$,
which implies that $\eta\in C_{nv}^{n}(\mathcal{E},\mathcal{H},\mathcal{V})$.

Thus $\omega=\eta+d\lambda,\ \eta\in C_{nv}^{n}(\mathcal{E},\mathcal{H},\mathcal{V}),\ \lambda\in C^{n-1}(\mathcal{E},\mathcal{H},\mathcal{V})$,
the proof is finished.
\end{proof}
Now we turn to the proof of Theorem \ref{thm:isomorphism CDA}:
\begin{proof}
By Lemma \ref{lem: CE nv}, we only need to prove that the inclusion
map
\[
\psi:C_{nv}(\mathcal{E},\mathcal{H},\mathcal{V})\rightarrow C(\mathcal{E},\mathcal{H},\mathcal{V})
\]
induces an isomorphism
\[
\psi_{*}:H_{nv}^{\bullet}(\mathcal{E},\mathcal{H},\mathcal{V})\rightarrow H^{\bullet}(\mathcal{E},\mathcal{H},\mathcal{V}).
\]

1). $\psi_{*}$ is surjective.

Given any $[\omega]\in H^{n}(\mathcal{E},\mathcal{H},\mathcal{V})$,
since $\omega$ is closed, by Lemma \ref{lem:lambda}, there exists
$\eta\in C_{nv}^{n}(\mathcal{E},\mathcal{H},\mathcal{V})$ and $\lambda\in C^{n-1}(\mathcal{E},\mathcal{H},\mathcal{V})$
such that $\omega=\eta+d\lambda$. We see that $\eta=\omega-d\lambda$
is closed, so $\psi_{*}([\eta])=[\eta]=[\omega]$.

2). $\psi_{*}$ is injective.

Suppose $\zeta\in C_{nv}^{n}(\mathcal{E},\mathcal{H},\mathcal{V})$
is exact in $C^{n}(\mathcal{E},\mathcal{H},\mathcal{V})$, i.e. $\exists\omega\in C^{n-1}(\mathcal{E},\mathcal{H},\mathcal{V}),\ d\omega=\zeta$,
we need to prove that $\zeta$ is exact in $C_{nv}^{n}(\mathcal{E},\mathcal{H},\mathcal{V})$.

Since $(d\omega)_{k}=\zeta_{k}=0,\ \forall k>0$, by Lemma \ref{lem:lambda},
there exists $\eta\in C_{nv}^{n-1}(\mathcal{E},\mathcal{H},\mathcal{V})$
and $\lambda\in C^{n-2}(\mathcal{E},\mathcal{H},\mathcal{V})$ such
that $\omega=\eta+d\lambda$. So 
\[
\zeta=d\omega=d\eta
\]
 is exact in $C_{nv}^{n}(\mathcal{E},\mathcal{H},\mathcal{V})$.

The proof is finished.\end{proof}
\begin{rem}
When $\rho^{*}$ is injective (in this case we call $\mathcal{E}$
a transitive Courant-Dorfman algebra), and $\mathcal{H}=\rho^{*}(\Omega^{1}),\ \mathcal{V}=R$,
Definition \ref{Def:generalized standard cohomology CDA} recovers
the ordinary standard cohomology (Definition \ref{Def:standard cohomology CDA}).
Moreover, if $\mathcal{E}=\Gamma(E)$ is the space of sections of
a transitive Courant algebroid $E$, Theorem \ref{thm:isomorphism CDA}
recovers the isomorphism between the standard cohomology and naive
cohomology of $E$, as conjectured by Stienon-Xu \cite{StienonXu08}
and first proved by Ginot-Grutzmann \cite{GinotGrutz08}. So Theorem
\ref{thm:isomorphism CDA} is a generalization of their result.\end{rem}
\begin{example}
Suppose $\mathcal{G}$ is a bundle of quadratic Lie algebras on $M$.
Given a standard Courant algebroid structure (see Chen, Stienon and
Xu \cite{ChenSX}) on 
\[
E=TM\oplus\mathcal{G}\oplus T^{*}M,
\]
let $\mathcal{E}=\Gamma(E)$. As mentioned in the remark above, if
we take $\mathcal{H}=\Gamma(T^{*}M)=\Omega^{1}(M)$ and $\mathcal{V}=C^{\infty}(M)$,
the $\Omega^{1}(M)$-standard cohomology $H^{\bullet}(\mathcal{E},\Omega^{1}(M),C^{\infty}(M))$
coincides with the standard cohomology of $E$, and is isomorphic
to the cohomology of Lie algebroid $TM\oplus\mathcal{G}$ with coefficients
in $C^{\infty}(M)$. 

Now suppose $\mathcal{K}$ is an isotropic ideal in $\mathcal{G}$,
then $\mathcal{H}=\Gamma(\mathcal{K}\oplus T^{*}M)\supseteq\Gamma(T^{*}M)$
is an isotropic ideal in $\mathcal{E}$. Given a $\Gamma(\mathcal{K}\oplus T^{*}M)$-representation
$\mathcal{V}$ (e.g. $C^{\infty}(M)$), we have the $\Gamma(\mathcal{K}\oplus T^{*}M)$-standard
cohomology $H^{\bullet}(\mathcal{E},\Gamma(\mathcal{K}\oplus T^{*}M),C^{\infty}(M))$.
By Theorem \ref{thm:isomorphism CDA}, it is isomorphic to the cohomology
of Lie algebroid $TM\oplus(\mathcal{G}/\mathcal{K})$ with coefficients
in $\mathcal{V}$.
\end{example}

\section{Crossed products of Leibniz algebras}

In this section, we associate a Courant-Dorfman algebra to any Leibniz
algebra and consider the relation between $H$-standard complexes
of them. At last we prove an isomorphism theorem for Leibniz algebras.

Given a Leibniz algebra $L$ with left center $Z$, let $S^{\bullet}(Z)$
be the algebra of symmetric tensors of $Z$. We construct a Courant-Dorfman
algebra structure on the tensor product
\[
\mathcal{L}\triangleq S^{\bullet}(Z)\otimes L
\]
 as follows:

let $R$ be $S^{\bullet}(Z)$;

let the $S^{\bullet}(Z)$-module structure of $\mathcal{L}$ be given
by multiplication of $S^{\bullet}(Z)$, i.e. 
\[
f_{1}\cdot(f_{2}\otimes e)\triangleq(f_{1}f_{2})\otimes e,\qquad\forall f_{1},f_{2}\in S^{\bullet}(Z),e\in L;
\]

(For simplicity, we will write $f\otimes e$ as $fe$ from now on.)

let the symmetric bilinear form $\langle\cdot,\cdot\rangle$ of $\mathcal{L}$
be the $S^{\bullet}(Z)$-bilinear extension of the symmetric product
$(\cdot,\cdot)$ of $L$, i.e. 
\[
\langle f_{1}e_{1},f_{2}e_{2}\rangle\triangleq f_{1}f_{2}(e_{1},e_{2}),\qquad\forall f_{1},f_{2}\in S^{\bullet}(Z),e_{1},e_{2}\in L
\]
(since $\langle e_{1},e_{2}\rangle=(e_{1},e_{2})$, in the following
we always use the notation $\langle\cdot,\cdot\rangle$); 

let the derivation $\partial:\ S^{\bullet}(Z)\rightarrow\mathcal{L}$
be the extension of the inclusion map $Z\hookrightarrow L$ by Leibniz
rule, i.e.
\[
\partial(z_{1}\cdots z_{k})\triangleq\sum_{1\leq i\leq k}(z_{1}\cdots\widehat{z_{i}}\cdots z_{k})\partial z_{i},\qquad\forall z_{i}\in Z;
\]

let the Dorfman bracket on $\mathcal{L}$, still denoted by $\circ$,
be the extension of the Leibniz bracket of $L$:
\begin{equation}
f_{1}e_{1}\circ f_{2}e_{2}\triangleq f_{1}f_{2}(e_{1}\circ e_{2})+\langle e_{1},e_{2}\rangle f_{2}\partial f_{1}+\langle e_{1},\partial f_{2}\rangle f_{1}e_{2}-\langle e_{2},\partial f_{1}\rangle f_{2}e_{1}\label{eq:bracket crossed product}
\end{equation}
$\forall f_{1},f_{2}\in S^{\bullet}(Z),e_{1},e_{2}\in L$.
\begin{prop}
With the above notations, $(\mathcal{L},S^{\bullet}(Z),\langle\cdot,\cdot\rangle,\partial,\circ)$
becomes a Courant-Dorfman algebra (called the crossed product of $L$).\label{prop: from Leibniz to CDA}\end{prop}
\begin{proof}
We need to check all the six conditions in Definition \ref{Def:Courant-Dorfman-algebra}.

1). $f_{1}e_{1}\circ f(f_{2}e_{2})=f(f_{1}e_{1}\circ f_{2}e_{2})+\langle f_{1}e_{1},\partial f\rangle f_{2}e_{2}$
\begin{eqnarray*}
 &  & f_{1}e_{1}\circ f(f_{2}e_{2})\\
 & = & ff_{1}f_{2}(e_{1}\circ e_{2})+\langle e_{1},e_{2}\rangle ff_{2}\partial f_{1}+\langle e_{1},\partial(ff_{2})\rangle f_{1}e_{2}-\langle e_{2},\partial f_{1}\rangle ff_{2}e_{1}\\
 & = & ff_{1}f_{2}(e_{1}\circ e_{2})+\langle e_{1},e_{2}\rangle ff_{2}\partial f_{1}+f\langle e_{1},\partial f_{2}\rangle f_{1}e_{2}+f_{2}\langle e_{1},\partial f\rangle f_{1}e_{2}-\langle e_{2},\partial f_{1}\rangle ff_{2}e_{1}\\
 & = & f(f_{1}e_{1}\circ f_{2}e_{2})+\langle f_{1}e_{1},\partial f\rangle f_{2}e_{2}.
\end{eqnarray*}

2). $\langle f_{1}e_{1},\partial\langle f_{2}e_{2},f_{3}e_{3}\rangle\rangle=\langle f_{1}e_{1}\circ f_{2}e_{2},f_{3}e_{3}\rangle+\langle f_{2}e_{2},f_{1}e_{1}\circ f_{3}e_{3}\rangle$
\begin{eqnarray*}
 &  & \langle f_{1}e_{1},\partial\langle f_{2}e_{2},f_{3}e_{3}\rangle\rangle\\
 & = & f_{1}\langle e_{1},\partial(f_{2}f_{3}\langle e_{2},e_{3}\rangle)\rangle\\
 & = & f_{1}f_{2}f_{3}\langle e_{1},\partial\langle e_{2},e_{3}\rangle\rangle+f_{1}f_{2}\langle e_{2},e_{3}\rangle\langle e_{1},\partial f_{3}\rangle+f_{1}f_{3}\langle e_{2},e_{3}\rangle\langle e_{1},\partial f_{2}\rangle\\
 & = & f_{1}f_{2}f_{3}\big(\langle e_{1}\circ e_{2},e_{3}\rangle+\langle e_{2},e_{1}\circ e_{3}\rangle\big)+f_{1}f_{2}\langle e_{2},e_{3}\rangle\langle e_{1},\partial f_{3}\rangle+f_{1}f_{3}\langle e_{2},e_{3}\rangle\langle e_{1},\partial f_{2}\rangle\\
 & = & f_{1}f_{2}f_{3}\langle e_{1}\circ e_{2},e_{3}\rangle+f_{2}f_{3}\langle e_{1},e_{2}\rangle\langle e_{3},\partial f_{1}\rangle+f_{1}f_{3}\langle e_{2},e_{3}\rangle\langle e_{1},\partial f_{2}\rangle-f_{2}f_{3}\langle e_{1},e_{3}\rangle\langle e_{2},\partial f_{1}\rangle\\
 &  & +f_{1}f_{2}f_{3}\langle e_{2},e_{1}\circ e_{3}\rangle+f_{2}f_{3}\langle e_{1},e_{3}\rangle\langle e_{2},\partial f_{1}\rangle+f_{1}f_{2}\langle e_{2},e_{3}\rangle\langle e_{1},\partial f_{3}\rangle-f_{2}f_{3}\langle e_{1},e_{2}\rangle\langle e_{3},\partial f_{1}\rangle\\
 & = & \langle f_{1}f_{2}(e_{1}\circ e_{2})+\langle e_{1},e_{2}\rangle f_{2}\partial f_{1}+\langle e_{1},\partial f_{2}\rangle f_{1}e_{2}-\langle e_{2},\partial f_{1}\rangle f_{2}e_{1},f_{3}e_{3}\rangle\\
 &  & +\langle f_{2}e_{2},f_{1}f_{3}(e_{1}\circ e_{3})+\langle e_{1},e_{3}\rangle f_{3}\partial f_{1}+\langle e_{1},\partial f_{3}\rangle f_{1}e_{3}-\langle e_{3},\partial f_{1}\rangle f_{3}e_{1}\rangle\\
 & = & \langle f_{1}e_{1}\circ f_{2}e_{2},f_{3}e_{3}\rangle+\langle f_{2}e_{2},f_{1}e_{1}\circ f_{3}e_{3}\rangle.
\end{eqnarray*}

3). $f_{1}e_{1}\circ f_{2}e_{2}+f_{2}e_{2}\circ f_{1}e_{1}=\partial\langle f_{1}e_{1},f_{2}e_{2}\rangle$
\begin{eqnarray*}
 &  & f_{1}e_{1}\circ f_{2}e_{2}+f_{2}e_{2}\circ f_{1}e_{1}\\
 & = & f_{1}f_{2}(e_{1}\circ e_{2})+\langle e_{1},e_{2}\rangle f_{2}\partial f_{1}+\langle e_{1},\partial f_{2}\rangle f_{1}e_{2}-\langle e_{2},\partial f_{1}\rangle f_{2}e_{1}\\
 &  & +f_{1}f_{2}(e_{2}\circ e_{1})+\langle e_{1},e_{2}\rangle f_{1}\partial f_{2}+\langle e_{2},\partial f_{1}\rangle f_{2}e_{1}-\langle e_{1},\partial f_{2}\rangle f_{1}e_{2}\\
 & = & f_{1}f_{2}\partial\langle e_{1},e_{2}\rangle+\langle e_{1},e_{2}\rangle f_{2}\partial f_{1}+\langle e_{1},e_{2}\rangle f_{1}\partial f_{2}\\
 & = & \partial\langle f_{1}e_{1},f_{2}e_{2}\rangle.
\end{eqnarray*}

Combining 1) and 3), we get the following:

\begin{eqnarray*}
 &  & f(f_{1}e_{1})\circ f_{2}e_{2}\\
 & = & (f(f_{1}e_{1})\circ f_{2}e_{2}+f_{2}e_{2}\circ f(f_{1}e_{1}))-f_{2}e_{2}\circ f(f_{1}e_{1})\\
 & = & \partial\langle f(f_{1}e_{1}),f_{2}e_{2}\rangle-(f(f_{2}e_{2}\circ f_{1}e_{1})+\langle f_{2}e_{2},\partial f\rangle f_{1}e_{1}\\
 & = & \langle f_{1}e_{1},f_{2}e_{2}\rangle\partial f+f\partial\langle f_{1}e_{1},f_{2}e_{2}\rangle-f(f_{2}e_{2}\circ f_{1}e_{1})-\langle f_{2}e_{2},\partial f\rangle f_{1}e_{1}\\
 & = & f(f_{1}e_{1}\circ f_{2}e_{2})+\langle f_{1}e_{1},f_{2}e_{2}\rangle\partial f-\langle f_{2}e_{2},\partial f\rangle f_{1}e_{1}
\end{eqnarray*}

4). $\langle\partial f,\partial f^{\prime}\rangle=0$.

We only need to consider the case of monomials, suppose
\[
f=z_{1}z_{2}\cdots z_{k},\ f^{\prime}=z_{1}^{\prime}z_{2}^{\prime}\cdots z_{l}^{\prime},\ z_{i},z_{j}^{\prime}\in Z,
\]
\begin{eqnarray*}
 &  & \langle\partial f,\partial f^{\prime}\rangle\\
 & = & \langle\sum_{i}(z_{1}\cdots\widehat{z_{i}}\cdots z_{k})\partial z_{i},\sum_{j}(z_{1}^{\prime}\cdots\widehat{z_{j}^{\prime}}\cdots z_{l}^{\prime})\partial z_{j}^{\prime}\rangle\\
 & = & \sum_{i,j}(z_{1}\cdots\widehat{z_{i}}\cdots z_{k}z_{1}^{\prime}\cdots\widehat{z_{j}^{\prime}}\cdots z_{l}^{\prime})(\partial z_{i}\circ\partial z_{j}^{\prime}+\partial z_{j}^{\prime}\circ\partial z_{i})\\
 & = & 0
\end{eqnarray*}

5). $\partial f\circ(f^{\prime}e)=0$

First we prove that $\partial f\circ e=0,\ \forall f\in S^{\bullet}(Z),\ e\in L$.

We only need to consider the case of monomials: suppose $f=z_{1}z_{2}\cdots z_{k},\ z_{i}\in Z$.

When $k=1$, i.e. $f=z_{1}\in Z$, the equation is trivial.

Now suppose the equation holds for any $k\leq m$, let's consider
the case of $k=m+1$.

\begin{eqnarray*}
 &  & \partial(z_{1}z_{2}\cdots z_{m+1})\circ e\\
 & = & \big((z_{1}\cdots z_{m})\partial z_{m+1}+z_{m+1}\partial(z_{1}\cdots z_{m})\big)\circ e\\
 & = & (z_{1}\cdots z_{m})(\partial z_{m+1}\circ e)+\langle\partial z_{m+1},e\rangle\partial(z_{1}\cdots z_{m})-\langle e,\partial(z_{1}\cdots z_{m})\rangle\partial z_{m+1}\\
 &  & +z_{m+1}(\partial(z_{1}\cdots z_{m})\circ e)+\langle\partial(z_{1}\cdots z_{m}),e\rangle\partial z_{m+1}-\langle e,\partial z_{m+1}\rangle\partial(z_{1}\cdots z_{m})\\
 & = & 0
\end{eqnarray*}

Thus by induction, $\partial f\circ e=0$ holds for any $f\in S^{\bullet}(Z)$.

Then combining 1) and 4),

\begin{eqnarray*}
\partial f\circ(f^{\prime}e) & = & f^{\prime}(\partial f\circ e)+\langle\partial f,\partial f^{\prime}\rangle e=0
\end{eqnarray*}

6). $f_{1}e_{1}\circ(f_{2}e_{2}\circ f_{3}e_{3})=(f_{1}e_{1}\circ f_{2}e_{2})\circ f_{3}e_{3}+f_{2}e_{2}\circ(f_{1}e_{1}\circ f_{3}e_{3})$

First we prove the equation for the case when $f_{2}=f_{3}=1$:

\begin{eqnarray*}
 &  & (f_{1}e_{1}\circ e_{2})\circ e_{3}+e_{2}\circ(f_{1}e_{1}\circ e_{3})\\
 & = & \big(f_{1}(e_{1}\circ e_{2})+\langle e_{1},e_{2}\rangle\partial f_{1}-\langle e_{2},\partial f_{1}\rangle e_{1}\big)\circ e_{3}+e_{2}\circ\big(f_{1}(e_{1}\circ e_{3})+\langle e_{1},e_{3}\rangle\partial f_{1}-\langle e_{3},\partial f_{1}\rangle e_{1}\big)\\
 & = & \big(f_{1}((e_{1}\circ e_{2})\circ e_{3})+\langle e_{1}\circ e_{2},e_{3}\rangle\partial f_{1}-\langle e_{3},\partial f_{1}\rangle(e_{1}\circ e_{2})\big)\\
 &  & +\big(\langle e_{1},e_{2}\rangle(\partial f_{1}\circ e_{3})+\langle\partial f_{1},e_{3}\rangle\partial\langle e_{1},e_{2}\rangle-\langle e_{3},\partial\langle e_{1},e_{2}\rangle\rangle\partial f_{1}\big)\\
 &  & -\big(\langle e_{2},\partial f_{1}\rangle(e_{1}\circ e_{3})+\langle e_{1},e_{3}\rangle\partial\langle e_{2},\partial f_{1}\rangle-\langle e_{3},\partial\langle e_{2},\partial f_{1}\rangle\rangle e_{1}\big)\\
 &  & +\big(f_{1}(e_{2}\circ(e_{1}\circ e_{3})+\langle e_{2},\partial f_{1}\rangle(e_{1}\circ e_{3})\big)+\big(\langle e_{1},e_{3}\rangle(e_{2}\circ\partial f_{1})+\langle e_{2},\partial\langle e_{1},e_{3}\rangle\rangle\partial f_{1}\big)\\
 &  & -\big(\langle e_{3},\partial f_{1}\rangle(e_{2}\circ e_{1})+\langle e_{2},\partial\langle e_{3},\partial f_{1}\rangle\rangle e_{1}\big)\\
 & = & f_{1}((e_{1}\circ e_{2})\circ e_{3})+f_{1}(e_{2}\circ(e_{1}\circ e_{3})+\big(\langle e_{1}\circ e_{2},e_{3}\rangle-\langle e_{3},\partial\langle e_{1},e_{2}\rangle\rangle+\langle e_{2},\partial\langle e_{1},e_{3}\rangle\rangle\big)\partial f_{1}\\
 &  & +\big(\langle e_{3},\partial\langle e_{2},\partial f_{1}\rangle\rangle-\langle e_{2},\partial\langle e_{3},\partial f_{1}\rangle\rangle\big)e_{1}+\langle e_{1},e_{3}\rangle\big(e_{2}\circ\partial f_{1}-\partial\langle e_{2},\partial f_{1}\rangle\big)\\
 &  & +\langle e_{3},\partial f_{1}\rangle\big(\partial\langle e_{1},e_{2}\rangle-e_{1}\circ e_{2}-e_{2}\circ e_{1}\big)+\langle e_{1},e_{2}\rangle(\partial f_{1}\circ e_{3})\\
 &  & +\langle e_{2},\partial f_{1}\rangle(e_{1}\circ e_{3})-\langle e_{2},\partial f_{1}\rangle(e_{1}\circ e_{3})\\
 & = & f_{1}(e_{1}\circ(e_{2}\circ e_{3}))+\big(\langle e_{2},\partial\langle e_{1},e_{3}\rangle\rangle-\langle e_{2}\circ e_{1},e_{3}\rangle\big)\partial f_{1}-\langle e_{2}\circ e_{3},\partial f_{1}\rangle e_{1}\\
 & = & f_{1}e_{1}\circ(e_{2}\circ e_{3})
\end{eqnarray*}

Then we prove the equation for the case when $f_{3}=1$:

\begin{eqnarray*}
 &  & (f_{1}e_{1}\circ f_{2}e_{2})\circ e_{3}+f_{2}e_{2}\circ(f_{1}e_{1}\circ e_{3})\quad\mbox{(let \ensuremath{x_{1}\triangleq f_{1}e_{1}})}\\
 & = & \big(f_{2}(x_{1}\circ e_{2})+\langle x_{1},\partial f_{2}\rangle e_{2}\big)\circ e_{3}+f_{2}(e_{2}\circ(x_{1}\circ e_{3}))+\langle e_{2},x_{1}\circ e_{3}\rangle\partial f_{2}-\langle x_{1}\circ e_{3},\partial f_{2}\rangle e_{2}\\
 & = & \big(f_{2}((x_{1}\circ e_{2})\circ e_{3})+\langle x_{1}\circ e_{2},e_{3}\rangle\partial f_{2}-\langle e_{3},\partial f_{2}\rangle(x_{1}\circ e_{2})\big)\\
 &  & +\big(\langle x_{1},\partial f_{2}\rangle(e_{2}\circ e_{3})+\langle e_{2},e_{3}\rangle\partial\langle x_{1},\partial f_{2}\rangle-\langle e_{3},\partial\langle x_{1},\partial f_{2}\rangle\rangle e_{2}\big)\\
 &  & +f_{2}(e_{2}\circ(x_{1}\circ e_{3}))+\langle e_{2},x_{1}\circ e_{3}\rangle\partial f_{2}-\langle x_{1}\circ e_{3},\partial f_{2}\rangle e_{2}\\
 & = & f_{2}((x_{1}\circ e_{2})\circ e_{3})+f_{2}(e_{2}\circ(x_{1}\circ e_{3}))+\langle x_{1},\partial f_{2}\rangle(e_{2}\circ e_{3})+\langle e_{2},e_{3}\rangle\partial\langle x_{1},\partial f_{2}\rangle\\
 &  & +\big(\langle x_{1}\circ e_{2},e_{3}\rangle+\langle e_{2},x_{1}\circ e_{3}\rangle\big)\partial f_{2}-\big(\langle e_{3},\partial f_{2}\rangle(x_{1}\circ e_{2})+(\langle e_{3},\partial\langle x_{1},\partial f_{2}\rangle\rangle+\langle x_{1}\circ e_{3},\partial f_{2}\rangle)e_{2}\big)\\
 & = & x_{1}\circ\big(f_{2}(e_{2}\circ e_{3})+\langle e_{2},e_{3}\rangle\partial f_{2}-\langle e_{3},\partial f_{2}\rangle e_{2}\big)\\
 & = & f_{1}e_{1}\circ(f_{2}e_{2}\circ e_{3}).
\end{eqnarray*}

Finally,

\begin{eqnarray*}
 &  & (f_{1}e_{1}\circ f_{2}e_{2})\circ f_{3}e_{3}+f_{2}e_{2}\circ(f_{1}e_{1}\circ f_{3}e_{3})\quad\mbox{(let \ensuremath{x_{1}\triangleq f_{1}e_{1},\ x_{2}\triangleq f_{2}e_{2}})}\\
 & = & f((x_{1}\circ x_{2})\circ e_{3})+\langle x_{1}\circ x_{2},\partial f_{3}\rangle e_{3}\\
 &  & +f(x_{2}\circ(x_{1}\circ e_{3}))+\langle x_{2},\partial f_{3}\rangle(x_{1}\circ e_{3})+\langle x_{1},\partial f_{3}\rangle(x_{2}\circ e_{3})+\langle x_{2},\partial\langle x_{1},\partial f_{3}\rangle\rangle e_{3}\\
 & = & x_{1}\circ\big(f_{3}(x_{2}\circ e_{3})+\langle x_{2},\partial f_{3}\rangle e_{3}\big)\\
 & = & f_{1}e_{1}\circ(f_{2}e_{2}\circ f_{3}e_{3})
\end{eqnarray*}

Thus the proposition is proved.
\end{proof}
By Equation \ref{eq:anchor CDA}, the anchor map
\[
\rho:\mathcal{L}\rightarrow Der(S^{\bullet}(Z),S^{\bullet}(Z))
\]
can be defined as follows:
\[
\rho(fe)(z_{1}\cdots z_{k})\triangleq f\sum_{i}(z_{1}\cdots\widehat{z_{i}}\cdots z_{k})(\rho(e)z_{i}),\quad\forall f\in S^{\bullet}(Z),e\in L,z_{i}\in Z.
\]
 
\begin{prop}
Suppose $H\supseteq Z$ is an isotropic ideal in $L$, and $(V,\tau)$
is an $H$-representation of $L$, let $\mathcal{V}\triangleq S^{\bullet}(Z)\otimes V$,
then

1). $\mathcal{H}\triangleq S^{\bullet}(Z)\otimes H$ is an isotropic
ideal in $\mathcal{L}$

2). $(V,\tau)$ induces an $\mathcal{H}$-representation $(\mathcal{V},\nabla)$
of $\mathcal{L}$, where $\nabla:\mathcal{L}\rightarrow Der(\mathcal{V})$
is defined as follows:
\[
\nabla_{f_{1}e}(f_{2}v)\triangleq f_{1}\big(\langle e,\partial f_{2}\rangle v+f_{2}(\tau(e)v)\big),\ \forall f_{1},f_{2}\in S^{\bullet}(Z),\ e\in L,\ v\in V.
\]
\end{prop}
\begin{proof}
1). Since 
\[
\langle f_{1}h_{1},f_{2}h_{2}\rangle=f_{1}f_{2}\langle h_{1},h_{2}\rangle=0,\quad\forall f_{1},f_{2}\in S^{\bullet}(Z),h_{1},h_{2}\in H,
\]
$\mathcal{H}$ is isotropic in $\mathcal{L}$. And it is easily observed
from Equation \ref{eq:bracket crossed product} that $\mathcal{H}$
is an ideal.

2). From the definition of $\nabla$, it is obvious that 
\begin{eqnarray*}
\nabla_{f_{1}h}(f_{2}v) & = & 0\\
\nabla_{f_{1}x}(f_{2}v) & = & f_{1}\nabla_{x}(f_{2}v)\\
\nabla_{x}(f_{1}(f_{2}v)) & = & (\rho(x)f_{1})(f_{2}v)+f_{1}\nabla_{x}(f_{2}v)
\end{eqnarray*}
$\forall f_{1},f_{2}\in S^{\bullet}(Z),h\in H,x\in\mathcal{L},v\in V.$
So we only need to prove that $\nabla$ is a homomorphism of Leibniz
algebras: 
\begin{eqnarray*}
 &  & [\nabla_{f_{1}e_{1}},\nabla_{f_{2}e_{2}}](fv)\\
 & = & \nabla_{f_{1}e_{1}}\big(f_{2}\langle e_{2},\partial f\rangle v+f_{2}f\tau(e_{2})v\big)-\nabla_{f_{2}e_{2}}\big(f_{1}\langle e_{1},\partial f\rangle v+f_{1}f\tau(e_{1})v\big)\\
 & = & \big(f_{1}\langle e_{1},\partial(f_{2}\langle e_{2},\partial f\rangle)\rangle v+f_{1}f_{2}\langle e_{2},\partial f\rangle\tau(e_{1})v+f_{1}\langle e_{1},\partial(f_{2}f)\rangle\tau(e_{2})v+f_{1}f_{2}f\tau(e_{1})\tau(e_{2})v\big)\\
 &  & -\big(f_{2}\langle e_{2},\partial(f_{1}\langle e_{1},\partial f\rangle)\rangle v+f_{2}f_{1}\langle e_{1},\partial f\rangle\tau(e_{2})v+f_{2}\langle e_{2},\partial(f_{1}f)\rangle\tau(e_{1})v+f_{2}f_{1}f\tau(e_{2})\tau(e_{1})v\big)\\
 & = & f_{1}f_{2}\big(\langle e_{1},\partial\langle e_{2},\partial f\rangle\rangle-\langle e_{2},\partial\langle e_{1},\partial f\rangle\rangle\big)v+f_{1}f_{2}f(\tau(e_{1})\tau(e_{2})v-\tau(e_{2})\tau(e_{1})v)\\
 &  & +f_{1}\langle e_{2},\partial f\rangle\langle e_{1},\partial f_{2}\rangle v-f_{2}\langle e_{1},\partial f\rangle\langle e_{2},\partial f_{1}\rangle v+f_{1}f\langle e_{1},\partial f_{2}\rangle\tau(e_{2})v-f_{2}f\langle e_{2},\partial f_{1}\rangle\tau(e_{1})v\\
 & = & f_{1}f_{2}\nabla_{e_{1}\circ e_{2}}(fv)+\langle e_{1},\partial f_{2}\rangle f_{1}\nabla_{e_{2}}(fv)-\langle e_{2},\partial f_{1}\rangle f_{2}\nabla_{e_{1}}(fv)\\
 & = & \nabla_{f_{1}f_{2}(e_{1}\circ e_{2})+\langle e_{1},\partial f_{2}\rangle f_{1}e_{2}-\langle e_{2},\partial f_{1}\rangle f_{2}e_{1}}(fv)\\
 & = & \nabla_{f_{1}e_{1}\circ f_{2}e_{2}}(fv)
\end{eqnarray*}

The proof is finished.
\end{proof}
Obviously $\mathcal{V}$ with the restriction of $\nabla$ to $L\subseteq\mathcal{L}$
is still an $H$-representation of $L$, we still denote it by $(V,\nabla)$. 

In the following, we always assume that 
\[
f\in S^{\bullet}(Z),\ e\in L,\ x\in\mathcal{L},\ h\in H,\ \alpha\in\mathcal{H}.
\]

\begin{thm}
\label{thm:CDA Leibniz standard}The $H$-standard complex of $L$
with coefficients in $\mathcal{V}$ is isomorphic to the $\mathcal{H}$-standard
complex of $\mathcal{L}$ with coefficients in $\mathcal{V}$, i.e.
\[
C^{\bullet}(L,H,\mathcal{V})\cong C^{\bullet}(\mathcal{L},\mathcal{H},\mathcal{V}).
\]
\end{thm}
\begin{proof}
For simplicity, let $(C_{1}^{\bullet},d_{1})$ be $C^{\bullet}(L,H,\mathcal{V})$,
and $(C_{2}^{\bullet},d_{2})$ be $C^{\bullet}(\mathcal{L},\mathcal{H},\mathcal{V})$.
Given any $\eta\in C_{2}^{n}$, we can obtain an associated cochain
in $C_{1}^{n}$ by restriction, denote this restriction map by $\psi$.
$\psi$ is obviously a cochain map.

Next, given any $\omega\in C_{1}^{n}$, we can extend it to a cochain
$\varphi\omega\in C_{2}^{n}$ as follows: 

for the degree 2 arguments, extend $\omega$ from $H$ to $\mathcal{H}$
by $S^{\bullet}(Z)$-linearity; 

for the degree 1 arguments, extend $\omega$ from $L$ to $\mathcal{L}$,
from the last argument to the first argument one by one, by the equation
of weak $S^{\bullet}(Z)$-linearity: 
\begin{eqnarray*}
 &  & (\varphi\omega)_{k}(e_{1},\cdots e_{a-1},fe_{a},x_{a+1},\cdots x_{n-2k};\alpha_{1},\cdots\alpha_{k})\\
 & = & f(\varphi\omega)_{k}(e_{1},\cdots e_{a-1},e_{a},x_{a+1},\cdots x_{n-2k};\alpha_{1},\cdots\alpha_{k})\\
 &  & +\sum_{b>a}(-1)^{b-a}\langle e_{a},x_{b}\rangle(\varphi\omega)_{k+1}(e_{1},\cdots e_{a-1},\widehat{e_{a}},x_{a+1},\cdots\widehat{x_{b}},\cdots x_{n-2k};\partial f,\alpha_{1},\cdots\alpha_{k}).
\end{eqnarray*}

The proof that $\varphi\omega$ is a cochain in $C_{2}^{n}$ is left
to the lemma below \ref{lem:welldefinedcochain}. 

Obviously, $\psi\circ\varphi=id_{C_{1}^{\bullet}},\ \varphi\circ\psi=id_{C_{2}^{\bullet}}$.
And $\varphi$ is also a cochain map:

\[
\varphi(d_{1}\omega)=\varphi(d_{1}(\psi(\varphi\omega)))=\varphi(\psi(d_{2}(\varphi\omega)))=d_{2}(\varphi\omega),\ \forall\omega\in C_{1}^{\bullet}
\]

The proof is finished.
\end{proof}
\vspace{3mm}

\begin{lem}
\label{lem:welldefinedcochain}$\eta\triangleq\varphi\omega$ as defined
above is a cochain in $C_{2}^{n}$.\end{lem}
\begin{proof}
$\eta$ is $S^{\bullet}(Z)$-linear in the arguments of $\mathcal{H}$
by definition. So we only need to prove weak skew-symmetricity and
weak $S^{\bullet}(Z)$-linearity in the arguments of $\mathcal{L}$.

Proof of weak skew-symmetricity: 
\begin{eqnarray}
 &  & \eta_{k}(x_{1},\cdots x_{a},x_{a+1}\cdots x_{n-2k};\alpha_{1},\cdots\alpha_{k})+\eta_{k}(x_{1},\cdots x_{a+1},x_{a}\cdots x_{n-2k};\cdots)\label{eq:1}\\
 & = & -\eta_{k+1}(x_{1},\cdots\widehat{x_{a}},\widehat{x_{a+1}},\cdots x_{n-2k};\partial\langle x_{a},x_{a+1}\rangle,\alpha_{1},\cdots\alpha_{k}).\nonumber 
\end{eqnarray}

Suppose $x_{b}=f_{b}e_{b},\ \forall b$. First we prove Equation \ref{eq:1}
for the case when $x_{1},\cdots x_{a-1}\in L$:
\begin{eqnarray}
 &  & \qquad\qquad\eta_{k}(e_{1},\cdots e_{a-1},x_{a},x_{a+1},\cdots x_{n-2k};\cdots)+\eta_{k}(e_{1},\cdots e_{a-1},x_{a+1},x_{a},\cdots x_{n-2k};\cdots)\label{eq:2}\\
 & = & f_{a}\eta_{k}(e_{1},\cdots e_{a},f_{a+1}e_{a+1},\cdots x_{n-2k};\cdots)-\langle e_{a},f_{a+1}e_{a+1}\rangle\eta_{k+1}(\cdots\widehat{e_{a}},\widehat{e_{a+1}},\cdots;\partial f_{a},\cdots)\nonumber \\
 &  & +\sum_{b>a+1}(-1)^{b+a}\langle e_{a},x_{b}\rangle\eta_{k+1}(\cdots\widehat{e_{a}},f_{a+1}e_{a+1},\cdots\widehat{x_{b}},\cdots;\partial f_{a},\cdots)\nonumber \\
 &  & +f_{a+1}\eta_{k}(\cdots e_{a+1},f_{a}e_{a},\cdots;\cdots)-\langle e_{a+1},f_{a}e_{a}\rangle\eta_{k+1}(\cdots\widehat{e_{a+1}},\widehat{e_{a}},\cdots;\partial f_{a+1},\cdots)\nonumber \\
 &  & +\sum_{b>a+1}(-1)^{b+a}\langle e_{a+1},x_{b}\rangle\eta_{k+1}(\cdots\widehat{e_{a+1}},f_{a}e_{a},\cdots\widehat{x_{b}},\cdots;\partial f_{a+1},\cdots)\nonumber \\
 & = & f_{a}f_{a+1}\eta_{k}(\dots e_{a},e_{a+1},\cdots)+f_{a}\sum_{b>a+1}(-1)^{b+a+1}\langle e_{a+1},x_{b}\rangle\eta_{k+1}(\cdots e_{a},\widehat{e_{a+1}},\cdots\widehat{x_{b}},\cdots;\partial f_{a+1},\cdots)\nonumber \\
 &  & +f_{a+1}\sum_{b>a+1}(-1)^{b+a}\langle e_{a},x_{b}\rangle\eta_{k+1}(\cdots\widehat{e_{a}},e_{a+1},\cdots\widehat{x_{b}},\cdots;\partial f_{a},\cdots)\nonumber \\
 &  & +\sum_{a+1<b<c}(-1)^{c+a}(-1)^{b+a+1}\langle e_{a},x_{c}\rangle\langle e_{a+1},x_{b}\rangle\eta_{k+2}(\cdots\widehat{e_{a}},\widehat{e_{a+1}},\cdots\widehat{x_{b}},\cdots\widehat{x_{c}},\cdots;\partial f_{a},\partial f_{a+1},\cdots)\nonumber \\
 &  & +\sum_{a+1<c<b}(-1)^{c+a}(-1)^{b+a}\langle e_{a},x_{c}\rangle\langle e_{a+1},x_{b}\rangle\eta_{k+2}(\cdots\widehat{e_{a}},\widehat{e_{a+1}},\cdots\widehat{x_{c}},\cdots\widehat{x_{b}},\cdots;\partial f_{a},\partial f_{a+1},\cdots)\nonumber \\
 &  & +f_{a}f_{a+1}\eta_{k}(\cdots e_{a+1},e_{a},\cdots)+f_{a+1}\sum_{b>a+1}(-1)^{b+a+1}\langle e_{a},x_{b}\rangle\eta_{k+1}(\cdots e_{a+1},\widehat{e_{a}},\cdots\widehat{x_{b}},\cdots;\partial f_{a},\cdots)\nonumber \\
 &  & +f_{a}\sum_{b>a+1}(-1)^{b+a}\langle e_{a+1},x_{b}\rangle\eta_{k+1}(\cdots\widehat{e_{a+1}},e_{a},\cdots\widehat{x_{b}},\cdots;\partial f_{a+1},\cdots)\nonumber \\
 &  & +\sum_{a+1<b<c}(-1)^{c+a}(-1)^{b+a+1}\langle e_{a+1},x_{c}\rangle\langle e_{a},x_{b}\rangle\eta_{k+2}(\cdots\widehat{e_{a+1}},\widehat{e_{a}},\cdots\widehat{x_{b}},\cdots\widehat{x_{c}},\cdots;\partial f_{a},\partial f_{a+1},\cdots)\nonumber \\
 &  & +\sum_{a+1<c<b}(-1)^{c+a}(-1)^{b+a}\langle e_{a+1},x_{c}\rangle\langle e_{a},x_{b}\rangle\eta_{k+2}(\cdots\widehat{e_{a+1}},\widehat{e_{a}},\cdots\widehat{x_{c}},\cdots\widehat{x_{b}},\cdots;\partial f_{a},\partial f_{a+1},\cdots)\nonumber \\
 &  & -\big(f_{a+1}\langle e_{a},e_{a+1}\rangle\eta_{k+1}(\cdots\widehat{e_{a}},\widehat{e_{a+1}},\cdots;\partial f_{a},\cdots)+f_{a}\langle e_{a+1},e_{a}\rangle\eta_{k+1}(\cdots\widehat{e_{a+1}},\widehat{e_{a}},\cdots;\partial f_{a+1},\cdots)\big)\nonumber \\
 & = & f_{a}f_{a+1}\big(\eta_{k}(\dots,e_{a},e_{a+1},\cdots;\cdots)+\eta_{k}(\dots,e_{a+1},e_{a},\cdots;\cdots)\big)\nonumber \\
 &  & -\big(f_{a+1}\langle e_{a},e_{a+1}\rangle\eta_{k+1}(\cdots\widehat{e_{a}},\widehat{e_{a+1}},\cdots;\partial f_{a},\cdots)+f_{a}\langle e_{a+1},e_{a}\rangle\eta_{k+1}(\cdots\widehat{e_{a+1}},\widehat{e_{a}},\cdots;\partial f_{a+1},\cdots)\big)\nonumber \\
 & = & -\eta_{k+1}(\cdots\widehat{e_{a}},\widehat{e_{a+1}},\cdots,x_{n-2k};\partial\langle f_{a}e_{a},f_{a+1}e_{a+1}\rangle,\cdots)\nonumber 
\end{eqnarray}

We will prove Equation \ref{eq:1} by mathematical induction.

If $n=2l$ is even, when $k=l-1$, Equation \ref{eq:1} is equivalent
to \ref{eq:2}.

If $n=2l+1$ is odd, when $k=l-1$, $\ref{eq:1}$ is the combination
of \ref{eq:2} and the following:

\begin{eqnarray*}
 &  & \eta_{l-1}(f_{1}e_{1},f_{2}e_{2},f_{3}e_{3};\alpha_{1},\cdots\alpha_{l-1})+\eta_{l-1}(f_{1}e_{1},f_{3}e_{3},f_{2}e_{2};\alpha_{1},\cdots\alpha_{l-1})\\
 & = & f_{1}\eta_{l-1}(e_{1},f_{2}e_{2},f_{3}e_{3};\cdots)-\langle e_{1},f_{2}e_{2}\rangle\eta_{l}(f_{3}e_{3};\partial f_{1},\cdots)+\langle e_{1},f_{3}e_{3}\rangle\eta_{l}(f_{2}e_{2};\partial f_{1},\cdots)\\
 &  & +f_{1}\eta_{l-1}(e_{1},f_{3}e_{3},f_{2}e_{2};\cdots)-\langle e_{1},f_{3}e_{3}\rangle\eta_{l}(f_{2}e_{2};\partial f_{1},\cdots)+\langle e_{1},f_{2}e_{2}\rangle\eta_{l}(f_{3}e_{3};\partial f_{1},\cdots)\\
 & = & -f_{1}\eta_{l}(e_{1};\partial\langle f_{2}e_{2},f_{3}e_{3}\rangle,\cdots)\\
 & = & -\eta_{l}(f_{1}e_{1};\partial\langle f_{2}e_{2},f_{3}e_{3}\rangle,\cdots).
\end{eqnarray*}

Now suppose Equation \ref{eq:1} holds for $k>m$, consider the case
when $k=m$. By Equation \ref{eq:2}, we can further suppose that
\ref{eq:1} holds for $x_{1},\cdots x_{i}\in L,\ i<a$. We will prove
\ref{eq:1} for the case when $k=m$ and $x_{1},\cdots,x_{i-1}\in L$:
\begin{eqnarray*}
 &  & \eta_{m}(e_{1},\cdots e_{i-1},f_{i}e_{i},\cdots x_{a},x_{a+1},\cdots;\cdots)+\eta_{m}(e_{1},\cdots e_{i-1},f_{i}e_{i},\cdots x_{a+1},x_{a},\cdots;\cdots)\\
 & = & f_{i}\eta_{m}(\cdots e_{i},\cdots x_{a},x_{a+1},\cdots)+\sum_{b>i\neq a,a+1}(-1)^{b-i}\langle e_{i},x_{b}\rangle\eta_{m+1}(\cdots\widehat{e_{i}},\cdots\widehat{x_{b}},\cdots x_{a},x_{a+1},\cdots;\partial f_{i}\cdots)\\
 &  & f_{i}\eta_{m}(\cdots e_{i},\cdots x_{a+1},x_{a},\cdots)+\sum_{b>i\neq a,a+1}(-1)^{b-i}\langle e_{i},x_{b}\rangle\eta_{m+1}(\cdots\widehat{e_{i}},\cdots\widehat{x_{b}},\cdots x_{a+1},x_{a},\cdots;\partial f_{i}\cdots)\\
 &  & +((-1)^{a-i}+(-1)^{a+1-i})\langle e_{i},x_{a}\rangle\eta_{m+1}(\cdots\widehat{e_{i}},\cdots\widehat{x_{a}},x_{a+1},\cdots;\partial f_{i}\cdots)\\
 &  & +((-1)^{a+1-i}+(-1)^{a-i})\langle e_{i},x_{a+1}\rangle\eta_{m+1}(\cdots\widehat{e_{i}},\cdots x_{a},\widehat{x_{a+1}},\cdots;\partial f_{i}\cdots)\\
 & = & f_{i}\big(\eta_{m}(\cdots e_{i},\cdots x_{a},x_{a+1},\cdots;\cdots)+\eta_{m}(\cdots e_{i},\cdots x_{a+1},x_{a},\cdots;\cdots)\big)\\
 &  & +\big(\sum_{b>i\neq a,a+1}(-1)^{b-i}\langle e_{i},x_{b}\rangle\eta_{m+1}(\cdots\widehat{e_{i}},\cdots\widehat{x_{b}},\cdots x_{a},x_{a+1},\cdots;\partial f_{i}\cdots)\\
 &  & \quad+\sum_{b>i\neq a,a+1}(-1)^{b-i}\langle e_{i},x_{b}\rangle\eta_{m+1}(\cdots\widehat{e_{i}},\cdots\widehat{x_{b}},\cdots x_{a+1},x_{a},\cdots;\partial f_{i}\cdots)\big)\\
 & = & -f_{i}\eta_{m+1}(\cdots e_{i},\cdots\widehat{x_{a}},\widehat{x_{a+1}},\cdots;\partial\langle x_{a},x_{a+1}\rangle,\cdots)\\
 &  & -\sum_{b>i\neq a,a+1}(-1)^{b-i}\langle e_{i},x_{b}\rangle\eta_{m+2}(\cdots\widehat{e_{i}},\cdots\widehat{x_{b}},\cdots\widehat{x_{a}},\widehat{x_{a+1}},\cdots;\partial\langle x_{a},x_{a+1}\rangle,\partial f_{i}\cdots)\\
 & = & -\eta_{m+1}(e_{1},\cdots e_{i-1},f_{i}e_{i},\cdots\widehat{x_{a}},\widehat{x_{a+1}},\cdots;\partial\langle x_{a},x_{a+1}\rangle,\cdots)
\end{eqnarray*}

By induction, \ref{eq:1} is proved.

Proof of weak $S^{\bullet}(Z)$-linearity: 
\begin{eqnarray}
 &  & \eta_{k}(x_{1},\cdots x_{i-1},fx_{i},\cdots x_{n-2k};\alpha_{1},\cdots\alpha_{k})\label{eq:3}\\
 & = & f\eta_{k}(\cdots x_{i},\cdots;\alpha_{1},\cdots\alpha_{k})+\sum_{a>i}(-1)^{a-i}\langle x_{i},x_{a}\rangle\eta_{k+1}(\cdots\widehat{x_{i}},\cdots\widehat{x_{a}},\cdots;\partial f,\alpha_{1},\cdots\alpha_{k})\nonumber 
\end{eqnarray}

When $x_{1},\cdots,x_{i-1}\in L$,\ref{eq:3} holds:
\begin{eqnarray*}
 &  & \eta_{k}(e_{1},\cdots,e_{i-1},ff_{i}e_{i},\cdots;\cdots)\\
 & = & ff_{i}\eta_{k}(\cdots e_{i},\cdots;\cdots)+\sum_{a>i}(-1)^{a-i}\langle e_{i},x_{a}\rangle\eta_{k+1}(\cdots\widehat{x_{i}},\cdots\widehat{x_{a}},\cdots;\partial(ff_{i}),\cdots)\\
 & = & f(\eta_{k}(\cdots f_{i}e_{i},\cdots;\cdots)-\sum_{a>i}(-1)^{a-i}\langle e_{i},x_{a}\rangle\eta_{k+1}(\cdots\widehat{x_{i}},\cdots\widehat{x_{a}},\cdots;\partial f_{i},\cdots))\\
 &  & +\sum_{a>i}(-1)^{a-i}\langle e_{i},x_{a}\rangle\eta_{k+1}(\cdots\widehat{x_{i}},\cdots\widehat{x_{a}},\cdots;\partial(ff_{i}),\cdots)\\
 & = & f\eta_{k}(\cdots f_{i}e_{i},\cdots;\cdots)+\sum_{a>i}(-1)^{a-i}\langle f_{i}e_{i},x_{a}\rangle\eta_{k+1}(\cdots\widehat{x_{i}},\cdots\widehat{x_{a}},\cdots;\partial f,\cdots)
\end{eqnarray*}

Now suppose \ref{eq:3} holds for any $k>m$, and for the case when
$x_{1},\cdots,x_{j}\in L(j<i),\ k=m$ as well. Consider the case when
$x_{1},\cdots,x_{j-1}\in L,\ k=m$:
\begin{eqnarray*}
 &  & \eta_{k}(e_{1},\cdots e_{j-1},f_{j}e_{j},\cdots,fx_{i},\cdots;\cdots)\\
 & = & f_{j}\eta_{k}(\cdots e_{j},\cdots fx_{i},\cdots;\cdots)+(-1)^{i-j}\langle e_{j},fx_{i}\rangle\eta_{k+1}(\cdots\widehat{e_{j}},\cdots\widehat{x_{i}},\cdots;\partial f_{j},\cdots)\\
 &  & +\sum_{b>j,b\neq i}(-1)^{b+j}\langle e_{j},x_{b}\rangle\eta_{k+1}(\cdots\widehat{e_{j}},\cdots\widehat{x_{b}},\cdots,fx_{i},\cdots;\partial f_{j},\cdots)\\
 & = & f_{j}\big(f\eta_{k}(\cdots e_{j},\cdots x_{i},\cdots;\cdots)+\sum_{a>i}(-1)^{a+i}\langle x_{i},x_{a}\rangle\eta_{k+1}(\cdots e_{j},\cdots\widehat{x_{i}},\cdots\widehat{x_{a}},\cdots;\partial f,\cdots)\big)\\
 &  & +\sum_{j<b<i}(-1)^{b+j}\langle e_{j},x_{b}\rangle\big(f\eta_{k+1}(\cdots\widehat{e_{j}},\cdots\widehat{x_{b}},\cdots,x_{i},\cdots;\partial f_{j},\cdots)\\
 &  & \quad+\sum_{a>i}(-1)^{a+i}\langle x_{i},x_{a}\rangle\eta_{k+2}(\cdots\widehat{e_{j}},\cdots\widehat{x_{b}},\cdots\widehat{x_{i}},\cdots,\widehat{x_{a}},\cdots;\partial f,\partial f_{j},\cdots)\big)\\
 &  & +\sum_{b>i}(-1)^{b+j}\langle e_{j},x_{b}\rangle\big(f\eta_{k+1}(\cdots\widehat{e_{j}},\cdots,x_{i},\cdots\widehat{x_{b}},\cdots;\partial f_{j},\cdots)\\
 &  & \quad+\sum_{i<a<b}(-1)^{a+i}\langle x_{i},x_{a}\rangle\eta_{k+2}(\cdots\widehat{e_{j}},\cdots\widehat{x_{i}},\cdots\widehat{x_{a}},\cdots\widehat{x_{b}},\cdots;\partial f,\partial f_{j},\cdots)\\
 &  & \quad+\sum_{b<a}(-1)^{a+i+1}\langle x_{i},x_{a}\rangle\eta_{k+2}(\cdots\widehat{e_{j}},\cdots\widehat{x_{i}},\cdots\widehat{x_{b}},\cdots\widehat{x_{a}},\cdots;\partial f,\partial f_{j},\cdots)\big)\\
 &  & +(-1)^{i+j}\langle e_{j},fx_{i}\rangle\eta_{k+1}(\cdots\widehat{e_{j}},\cdots\widehat{x_{i}},\cdots;\partial f_{j},\cdots)\\
 & = & f\big(f_{j}\eta_{k}(\cdots e_{j},\cdots,x_{i},\cdots;\cdots)+(-1)^{i+j}\langle e_{j},x_{i}\rangle\eta_{k+1}(\cdots\widehat{e_{j}},\cdots\widehat{x_{i}},\cdots;\partial f_{j},\cdots)\\
 &  & \quad+\sum_{j<b<i}(-1)^{b+j}\langle e_{j},x_{b}\rangle\eta_{k+1}(\cdots\widehat{e_{j}},\cdots\widehat{x_{b}},\cdots,x_{i},\cdots;\partial f_{j},\cdots)\\
 &  & \quad+\sum_{b>i}(-1)^{b+j}\langle e_{j},x_{b}\rangle\eta_{k+1}(\cdots\widehat{e_{j}},\cdots x_{i},\cdots\widehat{x_{b}},\cdots;\partial f_{j},\cdots)\big)\\
 &  & +\sum_{a>i}(-1)^{a+i}\langle x_{i},x_{a}\rangle\big(f_{j}\eta_{k+1}(\cdots\widehat{e_{j}},\cdots\widehat{x_{i}},\cdots\widehat{x_{a}},\cdots;\partial f,\cdots)\\
 &  & \quad+\sum_{j<b<i}(-1)^{b+j}\langle e_{j},x_{b}\rangle\eta_{k+2}(\cdots e_{j},\cdots\widehat{x_{b}},\cdots\widehat{x_{i}},\cdots x_{a},\cdots;\partial f,\partial f_{j},\cdots)\\
 &  & \quad+\sum_{i<b<a}(-1)^{b+j+1}\langle e_{j},x_{b}\rangle\eta_{k+2}(\cdots\widehat{e_{j}},\cdots\widehat{x_{i}},\cdots\widehat{x_{b}},\cdots\widehat{x_{a}},\cdots;\partial f,\partial f_{j},\cdots)\\
 &  & \quad+\sum_{b>a}(-1)^{b+j}\langle e_{j},x_{b}\rangle\eta_{k+2}(\cdots\widehat{e_{j}},\cdots\widehat{x_{i}},\cdots\widehat{x_{a}},\cdots\widehat{x_{b}},\cdots;\partial f,\partial f_{j},\cdots)\big)\\
 & = & f\eta_{k}(\cdots f_{j}e_{j},\cdots x_{i},\cdots;\cdots)+\sum_{a>i}(-1)^{a-i}\langle x_{i},x_{a}\rangle\eta_{k+1}(\cdots f_{j}e_{j},\cdots\widehat{x_{i}},\cdots,\widehat{x_{a}},\cdots;\partial f,\cdots)
\end{eqnarray*}

By mathematical induction, Equation \ref{eq:3} holds for any $k$.

The proof is finished.
\end{proof}
It is easily proved that the Chevalley-Eilenberg complex of the Lie
algebra $L/H$ with coefficients in $\mathcal{V}$ is isomorphic to
the Chevalley-Eilenberg complex of the Lie-Rinehart algebra $\mathcal{L}/\mathcal{H}$
with coefficients in $\mathcal{V}$. Whence, combining Theorem \ref{thm:CDA Leibniz standard}
and \ref{thm:isomorphism CDA} (the quotient $\mathcal{L}/\mathcal{H}\cong S^{\bullet}(Z)\otimes(L/H)$
is a free module), we have the following:
\begin{cor}
\label{cor:isomorphism Leibniz }With the notations above,
\[
H^{\bullet}(L,H,\mathcal{V})\cong H_{CE}^{\bullet}(L/H,\mathcal{V}).
\]

\end{cor}
Actually this result is true for the $H$-representation $(V,\tau)$:
\begin{thm}
With the notations above,
\[
H^{\bullet}(L,H,V)\cong H_{CE}^{\bullet}(L/H,V).
\]
\end{thm}
\begin{proof}
Consider a subspace of $C^{\bullet}(\mathcal{L},\mathcal{H},\mathcal{V})$:
\[
C_{0}^{\bullet}(\mathcal{L},\mathcal{H},\mathcal{V})\triangleq\bigoplus_{n}\{\omega\in C^{n}(\mathcal{L},\mathcal{H},\mathcal{V})|\omega_{k}(e_{1},\cdots e_{n-2k};h_{1},\cdots h_{k})\in V,\ \forall k,\ \forall e\in L,h\in H\}.
\]

It is easily checked from the definition of $d=d_{0}+\delta+d^{\prime}$
that $C_{0}^{\bullet}(\mathcal{L},\mathcal{H},\mathcal{V})$ is a
subcomplex.

Given any $\omega\in C_{0}^{\bullet}(\mathcal{L},\mathcal{H},\mathcal{V})$
satisfying the condition of Lemma \ref{lem:lambda}, we see that $\lambda$
and $\eta$ as constructed in the proof of this lemma are both in
$C_{0}^{\bullet}(\mathcal{L},\mathcal{H},\mathcal{V})$. Then by similar
arguments to the proof of Theorem \ref{thm:isomorphism CDA},
\[
H_{0}^{\bullet}(\mathcal{L},\mathcal{H},\mathcal{V})\triangleq H^{\bullet}(C_{0}^{\bullet}(\mathcal{L},\mathcal{H},\mathcal{V}),d)
\]
is isomorphic to the cohomology of the following subcomplex of $C^{\bullet}(\mathcal{L},\mathcal{H},\mathcal{V})$
\begin{eqnarray*}
C_{nv\ 0}^{\bullet}(\mathcal{L},\mathcal{H},\mathcal{V}) & \triangleq & \bigoplus_{n}\{\omega\in C^{n}(\mathcal{L},\mathcal{H},\mathcal{V})|\omega_{k}=0,\ \forall k\geq1,\ \iota_{\alpha}\omega_{0}=0,\ \forall\alpha\in\mathcal{H},\\
 &  & \quad\omega_{0}(e_{1},\cdots e_{n})\in V,\ \forall e\in L\},
\end{eqnarray*}
which is again isomorphic to $H_{CE}^{\bullet}(L/H,V)$.

On the other hand, in the proof of Theorem \ref{thm:CDA Leibniz standard},
if we restrict $\psi$ from $C^{\bullet}(\mathcal{L},\mathcal{H},\mathcal{V})$
to $C_{0}^{\bullet}(\mathcal{L},\mathcal{H},\mathcal{V})$ and $\varphi$
from $C^{\bullet}(L,H,\mathcal{V})$ to $C^{\bullet}(L,H,V)$, we
can get mutually invertible cochain maps between $C^{\bullet}(L,H,V)$
and $C_{0}^{\bullet}(\mathcal{L},\mathcal{H},\mathcal{V})$.

Thus 
\[
H^{\bullet}(L,H,\mathcal{V})\cong H_{0}^{\bullet}(\mathcal{L},\mathcal{H},\mathcal{V})\cong H_{CE}^{\bullet}(L/H,\mathcal{V}).
\]
\end{proof}
\begin{example}
If $L$ is the omni-Lie algebra $gl(V)\oplus V$, $V$ is the only
isotropic ideal of $L$ containing the left center $V$, and $(V,\tau)$
is a $V$-representation with $\tau$ being the standard action of
$gl(V)$ on $V$. As introduced by Weinstein \cite{Weinstein}, the
omni Lie algebra $gl(V)\oplus V$ can be viewed as the linearization
of the standard Courant algebroid $TV^{*}\oplus T^{*}V^{*}$, where
$gl(V)$ is identified with the space of linear vector fields and
$V$ is identified with the space of constant $1$-forms. If we ignore
the difference between $S^{\bullet}(V)$ and $C^{\infty}(V^{*})$,
the crossed product $\mathcal{L}$ as constructed in Proposition \ref{prop: from Leibniz to CDA}
can be viewed as a Courant-Dorfman subalgebra of $\Gamma(TV^{*}\oplus T^{*}V^{*})$,
in the sense that $\mathcal{L}$ consists of all polynomial vector
fields (excluding constant ones) and polynomial $1$-forms. By Theorem
\ref{thm:CDA Leibniz standard} and Corollary \ref{cor:isomorphism Leibniz },
the standard cohomology of $\mathcal{L}$ is isomorphic to the cohomology
of Lie algebra $gl(V)$ with coefficients in $S^{\bullet}(V)$, which
is trivial. Whence, although $\mathcal{L}$ is different from $\Gamma(TV^{*}\oplus T^{*}V^{*})$,
the standard cohomology of them are both trivial.\end{example}


\begin{thebibliography}{10}
\bibitem{Barnes} Barnes, D.W., $\emph{{\mbox{Some theorems on Leibniz algebras}}}$,
Communications in Algebra 39.7 (2011), 2463-2472.

\bibitem{Bloh65}Bloh, A., $\emph{{\mbox{On a generalization of the concept of Lie algebra}}}$,
Dokl. Akad. Nauk SSSR. Vol. 165. No. 3 (1965), 471-473.

\bibitem{ChenSX} Chen, Z., Stiénon, M. and Xu, P., $\emph{{\mbox{On regular Courant algebroids}}}$,
Journal of Symplectic Geometry 11.1 (2013), 1-24.

\bibitem{Covez} Covez, S., $\emph{{\mbox{The local integration of Leibniz algebras}}}$,
Annales de l'institut Fourier. Association des Annales de l'institut
Fourier, 63(1)(2013), 1-35.

\bibitem{GinotGrutz08} Ginot, G. and Grützmann M., $\emph{{\mbox{Cohomology of Courant algebroids with split base}}}$,
Journal of Symplectic Geometry 7.3 (2009), 311-335.

\bibitem{LiuWX} Liu, Z.J., Weinstein A. and Xu P., $\emph{{\mbox{Manin triples for Lie bialgebroids}}}$,
J. Differential Geom 45.3 (1997), 547-574.

\bibitem{Loday93} Loday, J.L., $\emph{{\mbox{Une version non commutative des algèbres de Lie}}}$,
L'Ens. Math. (2), 39 (1993), 269-293.

\bibitem{LodayPirash93} Loday, J.L. and Pirashvili T., $\emph{{\mbox{Universal enveloping algebras of Leibniz algebras and (co) homology}}}$,
Mathematische Annalen 296.1 (1993), 139-158.

\bibitem{Patsourakos} Patsourakos, A., $\emph{{\mbox{On Nilpotent Properties of Leibniz Algebras}}}$,
Communications in Algebra. 35 (12)(2007), 3828\textendash{}3834.

\bibitem{RoytPhD} Roytenberg, D., $\emph{{\mbox{Courant algebroids, derived brackets and even symplectic supermanifolds}}}$,
Ph.D. thesis, University of California, Berkeley, 1999.

\bibitem{RoytCDA} Roytenberg, D., $\emph{{\mbox{Courant\textendash Dorfman algebras and their cohomology}}}$,
Letters in Mathematical Physics 90.1-3 (2009), 311-351.

\bibitem{RoytNQmfd} Roytenberg, D., $\emph{{\mbox{On the structure of graded symplectic supermanifolds and Courant algebroids}}}$,
Contemporary Mathematics 315 (2002), 169-186.

\bibitem{ShengLiu13} Sheng, Y.H. and Liu Z.J., $\emph{{\mbox{Leibniz 2-algebras and twisted Courant algebroids}}}$,
Communications in Algebra 41.5 (2013), 1929-1953.

\bibitem{ShengLiu16} Sheng, Y.H. and Liu Z.J., $\emph{{\mbox{From Leibniz Algebras to Lie 2-algebras}}}$,
Algebras and Representation Theory 19.1 (2016), 1-5.

\bibitem{StienonXu08} Stiénon, M. and Xu P., $\emph{{\mbox{Modular classes of Loday algebroids}}}$,
Comptes Rendus Mathematique 346.3 (2008), 193-198.

\bibitem{Weinstein} Weinstein, A., $\emph{{\mbox{Omni-Lie Algebras}}}$,
arXiv preprint math.RT/9912190.\end{thebibliography}
\end{document}